\documentclass[12pt,a4paper]{article}

\usepackage[latin1]{inputenc}
\usepackage[english]{babel}
\usepackage{amsmath}
\usepackage{color}
\usepackage{float} 
\usepackage{amsfonts}
\usepackage{amssymb,amsthm}
\usepackage{hyperref} 
\usepackage{graphicx,abstract}

\newcommand{\PP}{\mathcal{P}}

\newcommand{\RR}{\mathbb{R}}
\newcommand{\Ri}{\mathcal{R}}
\newcommand{\Z}{\mathbb{Z}}
\newcommand{\R}{\mathbb{R}}
\newcommand{\ZZ}{\mathbb{Z}}
\newcommand{\T}{\mathcal{T}}
\newcommand{\NN}{\mathbb{N}}
\newcommand{\NNs}{\mathbb{N}^*}
\newcommand{\MM}{(\mathcal{M})}

\newcommand{\bd}{\mathrm{bd}\,}

\newcommand{\argmax}{\mathrm{argmax}\,}
\newcommand{\dom}{\mathrm{dom}\,}
\newcommand{\rpiz}{\RR/2\pi\Z \,}

\newcommand{\dist}{\mathrm{dist}}
\newcommand{\inte}{\mathrm{int}\,}

\DeclareMathOperator*{\argmin}{argmin}

\newtheorem{theorem}{Theorem}
\newtheorem*{theorem*}{Theorem}
\newtheorem{lemma}{Lemma}
\newtheorem{proposition}{Proposition}
\newtheorem{corollary}{Corollary}

\newtheorem{definition}{Definition}
\newtheorem{assumption}{Assumption}
\newtheorem{remark}{Remark}

\renewenvironment{proof}[1][]{\noindent {\bf Proof #1:\;}}{\hfill $\Box$}

\title{Curiosities and counterexamples in smooth convex optimization}
\author{
   J\'{e}r\^{o}me Bolte\footnote{Toulouse School of Economics, Universit\'{e} Toulouse 1 Capitole,   France.}   \, and Edouard Pauwels\footnote{IRIT, ANITI, Universit\'e de Toulouse, CNRS. DEEL, IRT Saint Exupery, Toulouse, France.}
}
\begin{document}
\date{Draft of \today}

\maketitle

\begin{abstract}
Counterexamples to some old-standing optimization  problems in the smooth convex coercive setting are provided. We show that block-coordinate, steepest descent with exact search or Bregman descent methods do not generally converge. Other  failures of various desirable features are established: directional convergence of Cauchy's gradient curves, convergence of   Newton's flow, finite length of Tikhonov path, convergence of  central paths, or smooth Kurdyka-\L ojasiewicz inequality. All examples are  planar.\\
These examples are based on general smooth convex interpolation results. Given a decreasing sequence of  positively curved $C^k$ convex compact sets in the plane, we provide a level set interpolation of a $C^k$ smooth convex function where $k\geq2$ is arbitrary. If the intersection is reduced to one point our interpolant has positive definite Hessian, otherwise it is positive definite out of the solution set. Furthermore, given a sequence of decreasing polygons we provide an interpolant agreeing with the vertices and whose  gradients coincide with prescribed normals. 
\end{abstract}
\newpage
\tableofcontents
\newpage

\section{Introduction}

\subsection{Questions and method} One of the goals of convex optimization is to provide a solution to a problem with  stable and fast algorithms. The quality of a method is generally assessed by the  convergence of sequences, rate estimates, complexity bounds, finite length of relevant quantities and  other quantitative or qualitative ways.

Positive results in this direction are numerous and have been the object of intense research since decades. To name but a few: gradient methods e.g., \cite{newmirovsky1983problem,Nesterov,Boyd}, proximal methods e.g., \cite{PLC}, alternating methods e.g., \cite{Beck,wright2015coordinate}, path following methods e.g., \cite{Aus99,NN}, Tikhonov regularization e.g. \cite{Golub}, semi-algebraic optimization e.g., \cite{jon,BNPS}, decomposition methods e.g., \cite{PLC,Beck}, augmented Lagrangian methods e.g., \cite{Bertsek} and many others.

Despite this vast literature, some simple questions remain unanswered or just partly tackled, even for {\em smooth convex coercive} functions.
Does the alternating minimization method, aka Gauss-Seidel method, converge? Does the steepest descent method with exact line search converge? Do mirror descent or Bregman methods converge? Does Newton's flow converge? Do central paths converge? Is the gradient flow directionally stable? Do smooth convex functions have the Kurdyka-\L ojasiewicz property?

In this article we provide a negative answer to all  these questions.

Our work draws inspiration from early works of de Finetti \cite{definetti1949stratificazioni}, Fenchel \cite{fenchel51}, on convex interpolation, but also from  Torralba's PhD thesis \cite{torralba96} and the more recent \cite{bolte2010characterization}, where some counterexamples on the Tikhonov path and \L ojasiewicz inequalities are provided. The convex interpolation problem, see   \cite{definetti1949stratificazioni}, is as follows: given a monotone sequence of convex sets\footnote{In the sense of sets inclusion the sequence being indexed on $\NN$ or $\ZZ$} may we find a convex function interpolating each of these sets, i.e., having these sets as sublevel sets? Answers to these questions for {\em continuous} convex functions were provided by de Finetti, and improved by Fenchel \cite{fenchel51}, Kannai \cite{kannai77}, and then used in \cite{torralba96,bolte2010characterization}  for building counterexamples. 

We  improve this work by providing, for $k\geq 2$ arbitrary, a general $C^k$ interpolation theo\-rem for positively curved convex sets, imposing at the same time the positive definiteness of its Hessian out of the solution set. An abridged version could be as follows.
\begin{theorem*}[Smooth  convex interpolation\footnote{See Theorem \ref{th:smoothinterp} for the full version.}]
Let  $\left( T_i \right)_{i\in \ZZ}$ be a sequence of compact convex subsets of $\R^2$, with positively curved $C^k$ boundary, such that $T_i\subset\inte T_{i+1}$ for all $i$ in $\ZZ$. Then there exists a $C^k$ convex function $f$ having the $T_i$ as sublevel sets with positive definite Hessian outside of the set:  $$\argmin f=\bigcap_{i\in\ZZ} T_i.$$
\end{theorem*}

We provide several additional tools (derivatives computations)  and variants (status of the solution set, Legendre functions, Lipschitz continuity). Whether our result is generalizable to general smooth convex sequences, i.e.,  with vanishing curvature,  seems to be a very delicate question whose answer might well be negative.

Our central theoretical result is complemented by a discrete approximate interpolation result ``of order one" which is particularly well adapted for building counterexamples. Given a nested collection of polygons, one can indeed build a smooth convex function having level sets interpolating its vertices and whose gradients agree with  prescribed normals.   

Our results are obtained by blending parametrization techniques, Minkovski summation, Bernstein approximations and convex analysis.

As sketched below, our results offer the possibility of building counterexamples in convex  optimization  by restricting oneself to the construction of  countable collections of nested convex sets satisfying some desirable properties. In all cases  failures of good properties are  caused by some curvature oscillations.

\subsection{A digest of counterexamples}
 Counterexamples provided in this article can be classified along three axes: structural counterexamples\footnote{By structural, we include  homotopic deformations by mere summation}, counterexamples for convex optimization algorithms and ordinary differential equations.

 In the following, the term ``nonconverging'' sequence or trajectory means, a sequence or a trajectory with at least two distinct accumulation points. Unless otherwise stated, convex functions  have full domain.

\medskip

{\em The following results are proved for  $C^k$ convex functions on the plane with $k\geq 2$.}
\paragraph{Structural counterexamples} 
\begin{itemize}
    \item[(i)] \textbf{Kurdyka-\L ojasiewicz:} There exists a $C^k$ convex function  whose Hessian is po\-si\-tive definite outside its solution set and which does not satisfy the Kurdyka-\L ojasiewicz inequality. This is an improvement on  \cite{bolte2010characterization}.
    \item[(ii)] \textbf{Tikhonov regularization path:} There exists a $C^k$ convex function $f$ such that the regularization path 
    \begin{align*}
        x(r)= \argmin\left\{ f(y) + r \|y\|^2:y\in \R^2\right\}, \,\,r\in(0,1)
    \end{align*}
    has infinite length. This  strengthens  a theorem by Torralba \cite{torralba96}.
    \item[(iii)] \textbf{Central path:} 
    There exists a continuous Legendre function $h \colon [-1,1]^2 \mapsto \RR$,  $C^k$ on the interior, and $c$ in $\RR^2$ such that the central path
    \begin{align*}
        x(r) =
    \argmin \left\{ \left\langle c, y \right\rangle + r h(y):y\in D\right\}
    \end{align*}
    does not have a limit as $r \to 0$.
\end{itemize}

\paragraph{Algorithmic counterexamples:}

\begin{itemize}
    \item[(iv)] \textbf{Gauss-Seidel method (block coordinate descent):} There exists a $C^k$ convex function with  positive definite Hessian outside its solution set and an initialization $ (u_0,v_0)$ in $\RR^2$, such that the alternating minimization algorithm
    \begin{align*}
        u_{i+1} &= \argmin_{u \in \RR} f(u, v_i) \\
        v_{i+1} &= \argmin_{v \in \RR}f(u_{i+1}, v) 
    \end{align*}
    produces a bounded nonconverging sequence $((u_i,v_i))_{i\in \NN}$.
    \item[(v)] \textbf{Gradient descent with exact line search:} There exists a $C^k$ convex function $f$ with  positive definite Hessian outside its solution set and an initialization $x_0$ in $\RR^2$, such that the gradient descent algorithm with exact line search
    \begin{align*}
        x_{i+1} &= \argmin_{t \in \RR} f(x_i + t \nabla f(x_i))
    \end{align*}
    produces a bounded nonconvergent sequence.
    \item[(vi)] \textbf{Bregman or mirror descent method:} There exists a continuous Legendre function $h \colon [-1,1]^2 \mapsto \RR$, $C^k$ on the interior, a vector $c$ in $\RR^2$ and an initialization $x_0$ in $(-1,1)^2$ such that the Bregman recursion
    \begin{align*}
        x_{i+1} = \nabla h^*(\nabla h(x_i) - c)
    \end{align*}
    produces a nonconverging sequence. The couple $(h,\langle c,\cdot\rangle$) satisfies the smoothness  properties introduced in \cite{bauschke2016descent}.
\end{itemize}

\paragraph{Continuous time ODE counterexamples:}
\begin{itemize}
    \item[(vii)] \textbf{Continuous time Newton method:} There exists a $C^k$ convex function with positive definite Hessian outside its solution set, and an initialization $x_0$ in $\RR^2$ such that the continuous time Newton's system
    \begin{align*}
        \dot{x}(t) &= - \left[\nabla^2 f(x(t))\right]^{-1} \nabla f(x(t)), \,\,t\geq 0,\\
        x(0) &= x_0
    \end{align*}
    has a  solution approaching $\argmin f$  which does not converge.
    \item[(viii)] \textbf{Directional convergence for gradient curves:} There exists a $C^k$ convex function with $0$ as a unique minimizer and a positive definite Hessian on $\R^2\setminus\{0\}$, such that for any non stationary solution to the gradient system
    \begin{align*}
        \dot{x}(t) &= - \nabla f(x(t)) 
    \end{align*}
    the direction $x(t) / \|x(t)\|$ does not converge.
    \item[(ix)] \textbf{Hessian Riemannian gradient dynamics:} There exists a continuous Legendre function $h \colon [-1,1]^2 \mapsto \RR$, $C^k$ on the interior, a linear function $f$ and a nonconvergent solution to the following system
    \begin{align*}
        \dot{x}(t) = - \nabla_H f(x(t)),
    \end{align*}
    where $H = \nabla^2 h$ is the Hessian of $h$ and $\nabla_H f = H^{-1}  \nabla f$ is the gradient of $f$ in the Riemannian metric induced by $H$ on $(-1,1)^2$.
\end{itemize}

\paragraph{Pathological sequences and curves}
Our counterexamples lead to sequences or paths in $\RR^2$ which are related to a function $f$ by a certain property (see examples above) and have a certain type of pathology. For illustration purposes, we provide sketches of the pathological behaviors we met in~Figure \ref{fig:pathologicalbehaviors}.
\begin{figure}[H]
    \centering
    \includegraphics[width = \textwidth]{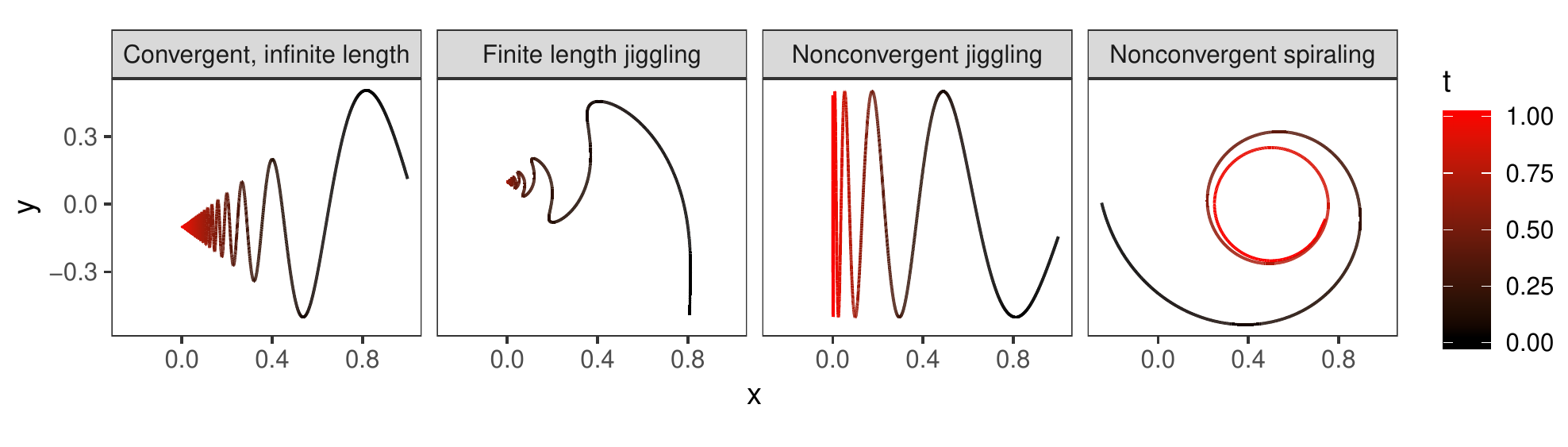}
    \caption{Rough sketches of pathological behavior for curves in the plane; for sequences similar figures would be obtained. Red colors indicate proximity with the solution set. Convergence with infinite length corresponds to counterexample (ii), finite length jiggling corresponds to counterexample (viii) (recall that gradient curves have to be self-contractant as well, see \cite{dan}), nonconvergent jiggling corresponds to counterexamples (iii), (vi), (vii), (ix) and nonconvergent spiraling corresponds to counterexamples (iv), (v) and somehow (i). }
    \label{fig:pathologicalbehaviors}
\end{figure}

\section{Preliminaries}

 Let us set $\NNs=\NN\setminus\{0\}$. For $p$ in $\NN^*$, the Euclidean scalar product in  $\R^p$ is denoted by $\langle\cdot,\cdot\rangle$, otherwise stated the norm is denoted by $\|\cdot\|$. 
Given subsets $S,T$ in $\R^p$, and $x$ in $\R^p$,  we define
$$\dist(x,S):=\inf \left\{\|x-y\|:y\in S\right\},$$ and the Hausdorff distance between $S$ and $T$,
$$\dist(S,T)=\max\left(\sup_{x\in S}\dist(x,T),\sup_{x\in T}\dist(x,S)\right).$$

Throughout this note, the assertion ``$g$ is $C^k$ on $D$'' where $D$ is not an open set is to be understood as ``$g$ is $C^k$ on an open neighborhood of $D$''. Given a map $G \colon X \mapsto A \times B$ for some space $X$, $[G]_1 \colon X \mapsto A$ denotes the first component of $G$.

\subsection{Continuous convex interpolations}
We consider a sequence of compact convex subsets of $\RR^p$, $\left(T_i \right)_{i \in \Z}$ such that  $T_{i+1} \subset \mathrm{int}\  T_i$. 
Finding a continuous convex interpolation of $\left(T_i \right)_{i \in \Z}$   is finding a convex continuous function which makes the $T_i$ a sequence of level sets. We call this process {\em continuous convex interpolation}. This questioning was present in Fenchel \cite{fenchel51} and dates back to de Finetti \cite{definetti1949stratificazioni}, let us also mention the work of Crouzeix \cite{crouzeix80} revolving around this issue.

Such constructions have been shown to be realizable by Torralba \cite{torralba96}, Kannai \cite{kannai77},  using ideas based on Minkowski sum. The validity of this construction can be proved easily using the result of Crouzeix \cite{crouzeix80} which was already present under  different and weaker forms in the works of de Finetti and Fenchel.

\begin{theorem}[de Finetti-Fenchel-Crouzeix]
				Let $f \colon \RR^p \to \RR$ be a quasiconvex function. The functions
				\begin{align*}
								F_x \colon \lambda \mapsto \sup\left\{  \langle z,x\rangle: f(z)\leq \lambda\right\}
				\end{align*}
				are concave for all $x$ in $\R^p$, if and only $f$ is convex.
				\label{th:crouzeix}
\end{theorem}

Our goal is to build smooth convex interpolation for sequences of smooth convex sets. To make such a construction we shall  use {\em nonlinear} {\em Minkowski} {\em interpolation} bet\-ween level sets. 

We shall also rely on Bernstein approximation which we now describe.

\subsection{Bernstein approximation}

We refer to the monograph \cite{lorentz1954bernstein} by G.G. Lorentz. 

The main properties of Bernstein polynomials to be used in this paper are the following:
\begin{itemize}
				\item[---] Bernstein approximation is linear in its functional argument $f$ and ``monotone'' which allows to construct an approximation using only positive combination of a finite number of function values.
				\item[---] There are precise formulas for derivatives of Bernstein approximation. They involve repeated finite differences. So approximating piecewise affine function with high enough degree leads to an approximation for which corner values of derivatives are controlled while the remaining derivatives are vanishing (up to a given order).
				\item[---] Bernstein approximation is shape preserving, which means in particular that approximating a concave function preserves concavity.
\end{itemize}

The main idea to produce a smooth interpolation which preserves level sets is depicted in  Figure \ref{fig:ideaBernstein} where we use Bernstein approximation to interpolate smoothly between three points  and controlling the successive derivatives at the end points of the interpolation.

\begin{figure}[H]
				\centering
				\includegraphics[width = .6\textwidth]{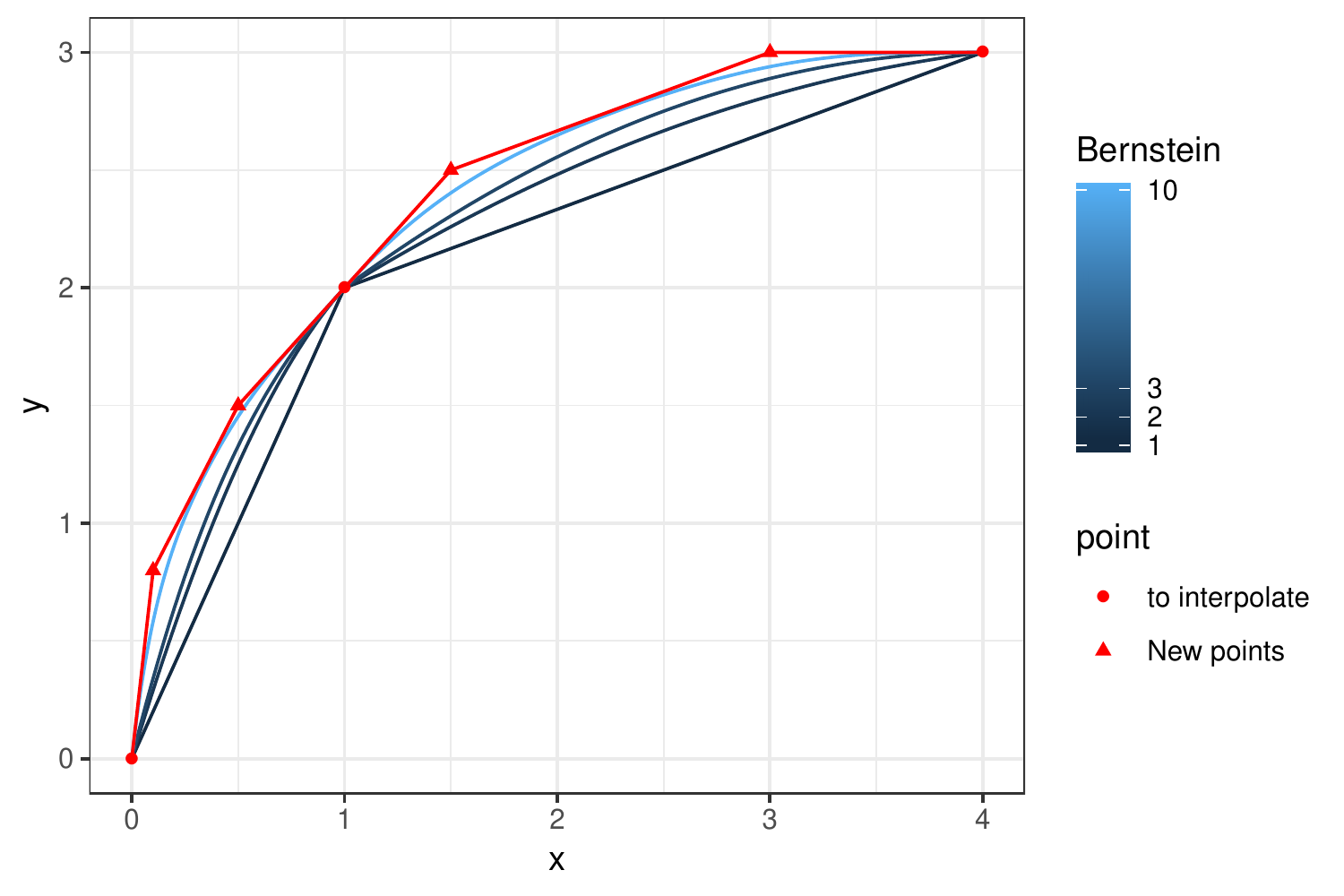}
				\caption{Illustration of Bernstein's smooth interpolation. We consider at first three points, to which we add four extra points to control derivatives at the junctions. We then build a concave, piecewise linear function and make a Bernstein approximation between the first two original points and the last two original points.}
				\label{fig:ideaBernstein}
\end{figure}

Let us now be specific. Given $f$ defined on the interval $[0,1]$, the {\em Bernstein polynomial} of order $d \in \NNs$ associated to $f$ is given by
\begin{align}
				B_d(x) = B_{d,f}(x) = \sum_{k=0}^d f\left( \frac{k}{d} \right) {d \choose k} x^k(1-x)^{d-k}, \mbox{ for $x\in [0,1].$}
				\label{eq:bernsteinPolynomial}
\end{align}

\paragraph{Derivatives and shape preservation:}
For any $h $ in $ (0,1)$ and $x $ in $ [0,1-h]$, we set $\Delta_h^1 f(x) = f(x+h) - f(x)$ and recursively for all $k $ in $ \NNs$, $\Delta_h^k f(x) = \Delta\left( \Delta_h^{k-1} f(x) \right)$. We fix $d \neq0$ in $\NN$ and for $h = \frac{1}{d}$ write $\Delta_h^k = \Delta^k$. Then for any $m \leq d$, we have
\begin{align}
				B_d(x)^{(m)} &= d(d-1)\ldots(d-m+1) \sum_{k=0}^{d - m} \Delta^m f\left( \frac{k}{d} \right) {{d-m} \choose k} x^k(1-x)^{d-k-m},
				\label{eq:bernsteinDerivative}
\end{align}
for any $x$ in $[0,1]$. If $f$ is increasing (resp. strictly increasing), then $\Delta^1f(x) \geq 0$ (resp. $\Delta^1f(x) > 0$) for all $x$ and $B'_d$ is positive (resp. strictly positive) and $B_d$ is increasing (resp. strictly increasing). Similarly, if $f$ is concave, then $\Delta^2 f(x) \leq 0$ for all $x$ so that $B^{(2)}_d \leq 0$ and $B_d$ is concave. 
From \eqref{eq:bernsteinDerivative}, we infer
\begin{equation}\label{e:born}
  \left|B_d(x)^{(m)}\right|\leq d(d-1)\ldots(d-m+1)\sup_{k\in\{0,\ldots,d-m\}}  \left|\Delta^m f\left( \frac{k}{d} \right)\right|
\end{equation}
for $x$ in $[0,1]$.
\paragraph{Approximation of  piecewise affine functions:} The following lemma will be extensively used throughout the proofs.
\begin{lemma}[Smoothing of piecewise lines in $\R^p$] \label{lem:convexInterpol}
				Let $q_0,q_1 \in \RR^p$, $\lambda_-<\lambda_0< \lambda_1 < \lambda_+$ and $0<e_1,e_0<1$. Set $\Theta = \left( q_0,q_1,,\lambda_-,\lambda_0,\lambda_1,\lambda_+, e_0,e_1 \right)$ and	define $\gamma_\Theta \colon [0,1] \mapsto \RR^p$ through 
				\begin{align*}\small
								\gamma_\Theta(t)=
								\begin{cases}
												q_0 \left( 1 + \frac{e_0}{\lambda_0 - \lambda_-} \left( t (\lambda_+- \lambda_-) \right) \right) & \text{ if } 0\leq t \leq \frac{\lambda_0 - \lambda_-}{\lambda_+ - \lambda_-} \\
												q_0(1 + e_0) \left( \frac{\lambda_1 - \lambda_- - t (\lambda_+ - \lambda_-)}{\lambda_1 - \lambda_0} \right) + q_1(1 - e_1) \left( \frac{ \lambda_- - \lambda_0  + t (\lambda_+ - \lambda_-)}{\lambda_1 - \lambda_0} \right)& \text{ if } \frac{\lambda_0 - \lambda_-}{\lambda_+ - \lambda_-} \leq t \leq \frac{\lambda_1 - \lambda_-}{\lambda_+ - \lambda_-} \\
												q_1\left( 1 + \frac{(t-1)(\lambda_+-\lambda_-)e_1}{\lambda_+ - \lambda_1} \right) & \text{ if } \frac{\lambda_1 - \lambda_-}{\lambda_+ - \lambda_-}  \leq t \leq 1.
								\end{cases}
				\end{align*}
				The curve $\gamma_\Theta$ in $\RR^{p+1}$ is the affine interpolant between the points $q_0$, $ (1 + e_0) q_0$, $(1 - e_1) q_1$ and $ q_1$. For any $m $ in $ \NN$, we choose  $d$ in $\NN^*$ such that
				\begin{align}
	\label{e:deg}	\frac{m}{d} \leq \min\left\{ \frac{\lambda_0 - \lambda_-}{\lambda_+ - \lambda_-}, 1 - \frac{\lambda_1 - \lambda_-}{\lambda_+ - \lambda_-} \right\}.
				\end{align}
				We consider a Bernstein-like reparametrization of $\tilde\gamma_\Theta $ given by 
				\begin{align*}
								\tilde\gamma_\Theta \colon [\lambda_-,\lambda_+] &\mapsto \RR^p\\
								\lambda & \mapsto \sum_{k=0}^d \tilde\gamma_\Theta \left( \frac{k}{d} \right){d \choose k} \left(\frac{\lambda - \lambda_-}{\lambda_+ - \lambda_-}  \right)^k\left(1 - \frac{\lambda - \lambda_-}{\lambda_+ - \lambda_-}  \right)^{d-k}.
				\end{align*}
				Then the following holds, for any $2 \leq l \leq m$, $\tilde\gamma_\Theta $ is $C^m$ and
				\begin{align*}
								\tilde\gamma_\Theta (\lambda_-)&=q_0&
								\tilde\gamma_\Theta (\lambda_+)&=q_1\\
								\tilde\gamma_\Theta '(\lambda_-)&=\frac{e_0}{\lambda_0- \lambda-} q_0&
								\tilde\gamma_\Theta '(\lambda_+)&=\frac{e_1}{\lambda_+ - \lambda_1} q_1\\
								\tilde\gamma_\Theta ^{(l)}(\lambda_-) &=0  &\tilde\gamma_\Theta ^{(l)}(\lambda_+) &= 0.
				\end{align*}
				Furthermore, if $\gamma_\Theta $ has monotone coordinates (resp. strictly monotone, resp. concave, resp. convex), then so has $\tilde\gamma_\Theta $.
				
\end{lemma}
\begin{proof}
				Note that the dependence of $\tilde\gamma_\Theta $ in $(q_0,q_1)$ is linear so that the dependence of $\tilde\gamma_\Theta $ in $(q_0,q_1)$ is also linear. Hence $\tilde\gamma_\Theta $ is of the form $\lambda \mapsto a(\lambda) q_0 + b(\lambda)q_1$.  
				We can restrict ourselves to $p = 1$ since the general case follows from the univariate case applied coordinatewise.

				If $p=1$, then $\tilde\gamma_\Theta  = B_{f_{\Theta,d}} \circ A$, where $A \colon \lambda \mapsto \frac{\lambda - \lambda_-}{\lambda_+-\lambda_-}$. We have $\tilde\gamma_\Theta (0) = q_0$, $\tilde\gamma_\Theta (1) = q_1$ and $\Delta^1 f_{\Theta}(0) = \frac{\lambda_+ - \lambda_-}{\lambda_0 - \lambda_-}\frac{e_0}{d} q_0$, $\Delta^1 f_{\Theta}\left( 1 - \frac{1}{d} \right) = \frac{\lambda_+ - \lambda_-}{\lambda_+ - \lambda_1}\frac{e_1}{d} q_1$ and $\Delta^{(l)} f_{\Theta}(0) = \Delta^{(l)} f_{\Theta}\left( 1 - \frac{l}{d} \right) =0$. The results follow from the expressions in \eqref{eq:bernsteinPolynomial} and \eqref{eq:bernsteinDerivative} and the chain rule for $\tilde\gamma_\Theta  = B_{f_{\Theta,d}} \circ A$.

				The last property of $\tilde\gamma_\Theta $ is due to the shape preserving property of Bernstein approximation and the fact that $\tilde\gamma_\Theta  = B_{f_{\Theta,d}} \circ L$.
\end{proof}

\begin{remark}\label{rem:alignedInterpolation}
{\rm (a) {\bf [Affine image]} Using the notation of Lemma \ref{lem:convexInterpol}, if $(\lambda_-,q_0)$, $(\lambda_0, (1 + e_0) q_0)$, $(\lambda_1, (1 - e_1) q_1)$ and $(\lambda_+, q_1)$ are aligned, then the interpolation is actually an affine function.	\\
(b) {\bf [Degree of the interpolants]} Observe that the degree of the Bernstein interpolant is connected to the slopes of the piecewise path $\lambda$ by \eqref{e:deg}.}
\end{remark}

\section{Smooth convex interpolation}

Being given a subset $S$ of  $\R^p$, we denote by $\mathrm{int}(S)$ its  interior, $\bar S$ its  closure and $\bd S=\bar S\setminus \mathrm{int}(S)$  its boundary. Let us recall that the {\em support function} of $S$ is defined through
$$\sigma_S(x)=\sup\left\{\langle y,x\rangle :y\in S\right\}\in \R\cup\{+\infty\}.$$

\subsection{Smooth parametrization of convex rings}

A {\em convex ring} is a set of the form $C_1\setminus C_2$ where $C_1\subset C_2$ are convex sets. Providing adequate parameterizations for such objects is key for interpolating $C_1$ and $C_2$ by some (regular) convex function.

The following assertion plays a fundamental role.
\begin{assumption}
				Let $T_-,T_+ \subset \RR^2$ be convex, compact with $C^k$ boundary ($k \geq 2$) and positive curvature. Assume that, $T_- \subset \mathrm{int} (T_+)$ and $0 \in \mathrm{int}(T_-)$. 
				\label{ass:curvature}
\end{assumption}
The positive curvature assumption ensures that the boundaries can be parametrized by their normal, that is, for $i=-,+$, there exists 
\begin{align*}
				c_i \colon \RR / 2 \pi \Z  \mapsto \bd(T_i)
\end{align*}
such that the normal to $T_i$ at $c_i(\theta)$ is the vector $n(\theta) = (\cos(\theta),\sin(\theta))^T$ and $\dot{c}_i(\theta) = \rho_{i}(\theta) \tau(\theta)$ where $\rho_i > 0$ and $\tau(\theta) = (-\sin(\theta),\cos(\theta))$. In this setting, it holds that $c_i(\theta) = \mbox{argmax}_{y \in T_i} \left\langle n(\theta), y\right\rangle$. The map $c_i$ is the inverse of the Gauss map and is $C^{k-1}$ (see \cite{schneider1993convex} Section 2.5).

\begin{lemma}[Minkowski sum of convex sets with positive curvature]\label{l:curv}
Let $T_-,T_+$ be as in Assumption \ref{ass:curvature} with  normal parametrizations as above. For $a,b\geq0$ with $a+b>0$ set $T=aT_-+bT_+$.

Then $T$ has positive curvature and its boundary is given by
$$\bd T=\{a\,c_-(\theta)+b\,c_+(\theta): \theta \in  \RR / 2 \pi \Z \},$$
with the natural parametrization $\rpiz \ni \theta \to a\,c_-(\theta)+b\,c_+(\theta).$
\end{lemma}
\begin{proof} We may assume $ab>0$ otherwise the result is obvious. Let $x$ be in $\bd T$ and denote by $n(\theta)$ the  normal vector at $x$ for a well chosen $\theta$, so that $x=\argmax \left\{\langle y,n(\theta)\rangle\right\}: y \in T\}$. Observe that the definition of the Minkowski sum yields
\begin{eqnarray*}
\max_{y\in T}\langle y,n(\theta)\rangle & = & \max_{(v,w)\in T_-\times T_+}\langle a v+bw,n(\theta)\rangle \\
&=& a \max_{v\in T_-} \langle v,n(\theta)\rangle+b\max_{w\in T_+} \langle w,n(\theta)\rangle 
 \end{eqnarray*}
 so that 
 \begin{eqnarray*}
\langle x,n(\theta)\rangle & =& a \langle c_-(\theta),n(\theta)\rangle+b\langle c_+(\theta),n(\theta)\rangle \\
& = &  \langle a c_-(\theta)+b c_+(\theta),n(\theta)\rangle \end{eqnarray*}
 which implies by extremality of $x$ that $x=  a c_-(\theta)+b c_+(\theta)$. Conversely, for any such $x$, $n(\theta)$ defines a supporting hyperplane to $T$ and $x$ must be on the boundary of $T$. The other results follow immediately.
  \end{proof} 
  
  In the  following fundamental proposition, we provide a smooth parametrization of the convex ring $T_+\setminus \inte T_-$. The major difficulty is to control tightly the derivatives at the boundary so that the parametrizations can be glued afterward to build smooth  interpolants. 
\begin{proposition}[$C^k$ parametrization of convex rings]
                Let $T_-,T_+$ be as in Assumption~\ref{ass:curvature} with their normal parametrization as above.
				Fix $k\geq 2$,  $\lambda_-< \lambda_0 < \lambda_1 < \lambda_+$ and $0 < e_0,e_1<1$. Choose $d $ in $ \NN^*$, such that
				\begin{align*}
								\frac{k}{d} \leq \min\left\{ \frac{\lambda_0 - \lambda_-}{\lambda_+ - \lambda_-}, 1 - \frac{\lambda_1 - \lambda_-}{\lambda_+ - \lambda_-} \right\}.
				\end{align*}
				Consider the map
				\begin{align}
								G \colon [\lambda_-,\lambda_+] \times \RR/2\pi\Z  &\mapsto \RR^2 \nonumber\\
								(\lambda, \theta) \mapsto \tilde\gamma_\Theta(\lambda)
								\label{eq:mapG}
				\end{align}
				with $\Theta = (c_-(\theta), c_+(\theta), \lambda_-, \lambda_0, \lambda_1, \lambda_+, e_0,e_1)$ and $\tilde\gamma$ as given by Lemma \ref{lem:convexInterpol}. Assume further that:
\begin{center}
$\MM$\, For any $\theta \in \RR / 2\pi\Z $, $\lambda \mapsto \left\langle G(\lambda, \theta), n(\theta)\right\rangle$ has strictly positive derivative on $[\lambda_-,\lambda_+]$.
\end{center}
				
				Then the image of $G$ is $\Ri:=T_+ \setminus \mathrm{int}(T_-)$, $G$ is $C^k$ and satisfies, for any $2 \leq l \leq k$ and any $m $ in $ \NN^*$,
				\begin{align*}
								\frac{\partial^m G}{\partial \theta^m}(\lambda_-,\theta) &= c_-^{(m)}(\theta)&
								\frac{\partial^m G}{\partial \theta^m}(\lambda_+,\theta) &= c_+^{(m)}(\theta)\\
								\frac{\partial^{m+1} G}{\partial \lambda\partial \theta^m} (\lambda_-,\theta) &= c_-^{(m)}(\theta) \frac{e_0}{\lambda_0 - \lambda_-}&
								\frac{\partial^{m+1} G}{\partial \lambda\partial \theta^m} (\lambda_+,\theta) &= c_+^{(m)}(\theta) \frac{e_1}{\lambda_+-\lambda_1}\\
								\frac{\partial^{l+m} G}{\partial \lambda^l \partial \theta^m }(\lambda_-,\theta) &= 0&
								\frac{\partial^{l+m} G}{\partial \lambda^l \partial \theta^m }(\lambda_+,\theta) &= 0.
				\end{align*}
				 Besides $G$ is a diffeomorphism from its domain on to its image. Set $\Ri\ni x \mapsto (f(x), \theta(x))$ to be the inverse of $G$.
				Then $f$ is $C^k$ and in addition, for all $x$ in $ \RR$, 
				\begin{align}
					\nabla f(x) =\quad& \frac{1}{\left\langle \frac{\partial G}{\partial \lambda} (f(x), \theta(x)), n(\theta(x)) \right\rangle} \, n(\theta(x))\nonumber\\
				    \nabla \theta(x) =\quad& \frac{1}{\left\langle\frac{\partial G}{\partial \theta}(f(x),\theta(x)),\tau(\theta(x))\right\rangle} \tau(\theta(x)) \nonumber\\
				    &-  \frac{\left\langle \frac{\partial G}{\partial \lambda} (f(x),\theta(x)), \tau(\theta(x))\right\rangle}{\left\langle \frac{\partial G}{\partial \lambda} (f(x),\theta(x)), n(\theta(x))\right\rangle \left\langle\frac{\partial G}{\partial \theta}(f(x),\theta(x)),\tau(\theta(x))\right\rangle   } n(\theta(x)) \nonumber\\
				    \nabla^2 f(x)=\quad&\frac{\left\langle\frac{\partial G}{\partial \theta}(f(x),\theta(x)),\tau(\theta(x))\right\rangle}{\left\langle \frac{\partial G}{\partial \lambda} (f(x), \theta(x)), n(\theta(x)) \right\rangle}
        			         \nabla \theta(x) \nabla \theta (x)^T \nonumber\\
        			         &-   \frac{\left\langle  \frac{\partial^2 G}{\partial \lambda^2}(f(x),\theta(x)), n(\theta(x))\right\rangle}{\left\langle \frac{\partial G}{\partial \lambda} (f(x), \theta(x)), n(\theta(x)) \right\rangle} \nabla f(x) \nabla f(x)^T,
        			\label{eq:gradient} 
				\end{align}
			    where all denominators are positive.
                \label{prop:diffeoMorphism}
				
\end{proposition}

\begin{remark}{\rm 
Note that $G$ is actually well defined and smooth on an open set containing its domain. As we shall see it is also a diffeomorphism from an open set containing its domain onto its image.}
 \end{remark}

\begin{proof}
Note that by construction, we have $G(\lambda,\theta) = a(\lambda) c_-(\theta) + b(\lambda)c_+(\theta)$ for some polynomials $a$ and $b$ which are nonnegative on $[\lambda_-,\lambda_+]$. The formulas for the derivatives follow easily from this remark, the form of $a$ and $b$ and Lemma \ref{lem:convexInterpol}.

				Set, for any $\lambda $ in $ [\lambda_-,\lambda_+]$, $T_\lambda = a(\lambda) T_- + b(\lambda) T_+$. The resulting set $T_\lambda$ is convex and has a positive curvature by Lemma~\ref{l:curv}, and for $\lambda$ fixed $G(\lambda,\cdot)$ is the inverse of the Gauss map of $T_\lambda$, which constitutes a parametrization by normals of the boundary.

				Assume that $\lambda < \lambda'$, using the monotonicity assumption $\MM$, we have for any $\theta,\theta'$,
				\begin{align*}
								\left\langle n(\theta'), G(\lambda,\theta) \right\rangle& \leq \sup_{y \in T_{\lambda}} \left\langle n(\theta'),y \right\rangle \\
								& = \left\langle n(\theta'), G(\lambda,\theta') \right\rangle\\
								&< \left\langle n(\theta'), G(\lambda',\theta') \right\rangle
				\end{align*}
				so that $G(\lambda,\theta) \neq G(\lambda',\theta')$. Furthermore, we have by definition of $G(\lambda',\theta')$
				\begin{align*}
								G(\lambda,\theta) \in \bigcap_{\theta' \in \RR/2\pi\Z } \left\{ y,\; \left\langle y, n(\theta')\right\rangle \leq \left\langle n(\theta'), G(\lambda',\theta') \right\rangle \right\} = T_{\lambda'},	
				\end{align*}
				where the equality follows from the convexity of $T_{\lambda'}$.
				By convexity and compactness, this entails that $T_\lambda = \mathrm{conv}(\bd(T_\lambda)) \subset \inte T_{\lambda'}$.

				Let us show that the map $G$ is bijective, first consider proving surjectivity. 
				Let $f$ be defined on $T_+ \setminus \mathrm{int}(T_-)$ through
				\begin{align}
								f \colon x \mapsto \inf_{}\left\{ \lambda : \lambda\geq \lambda_-, \; x \in T_\lambda = a(\lambda) T_- + b(\lambda) T_+\right\}.
			\label{eq:defIntropol}
				\end{align}
				Since $a(\lambda_+)=0,  b(\lambda_+)=1$ this function is well defined and by compactness and continuity the infimum is achieved. It must hold that $x $ belongs to $ \bd(T_{f(x)})$, indeed, if $f(x) = \lambda_-$, then $x $ belongs to $ \bd(T_-)$ and otherwise, if $x $ is in $ \mathrm{int}(T_{\lambda'})$ for $\lambda' > \lambda_-$, then $f(x) < \lambda'$.
				 We deduce that $x$ is of the form $G(f(x),\theta)$ for a certain value of $\theta$, so that $G$ is surjective.

				As for injectivity, we have already seen a first case, the monotonicity assumption $\MM$  ensures that $\lambda \neq \lambda'$ implies $G(\lambda,\theta) \neq G(\lambda',\theta')$ for any $\theta,\theta'$. Furthermore, we have the second case, for any $\lambda $ in $ [\lambda_-,\lambda_+]$ and any $\theta$, $G(\lambda,\theta) = \arg\max_{y\in T_\lambda} \left\langle y,n(\theta) \right\rangle$ so that $\theta \neq \theta'$ implies $G(\lambda,\theta) \neq G(\lambda,\theta')$. So in all cases, $(\lambda,\theta) \neq (\lambda',\theta')$ implies that $G(\lambda,\theta) \neq G(\lambda',\theta')$ and $G$ is injective.

				Let us now show that the map $G$ is a local diffeomorphism by estimating its Jacobian map. 
				
				Since $0 \in \mathrm{int}(T_-)$, we have for any $\lambda,\theta$,
				\begin{align*}
								0 &< \sup_{y \in T_-} \left\langle y, n(\theta) \right\rangle \\								
								&= \left\langle  G(\lambda_-,\theta),n(\theta) \right\rangle \\								
								& \leq a(\lambda) \left\langle  c_-(\theta), n(\theta)\right\rangle + b(\lambda) \left\langle c_+(\theta), n(\theta)  \right\rangle,
				\end{align*}
				and both scalar products are positive so that $a(\lambda) + b(\lambda)  > 0$. Hence, for any $\theta $ in $ \RR / 2\pi \Z $,
				\begin{align}
				\label{eq:derivativeGtheta}
								\frac{\partial G}{\partial \theta}(\lambda,\theta) = (a(\lambda) \rho_-(\theta) + b(\lambda) \rho_+ (\theta)) \tau(\theta),
				\end{align}
				with $a(\lambda) \rho_-(\theta) + b(\lambda) \rho_+ (\theta) > 0$. Furthermore by assumption $\lambda\to  \max_{y \in T_\lambda} \left\langle y, n(\theta) \right\rangle=\left\langle  G(\lambda,\theta),n(\theta) \right\rangle$ has strictly positive derivative in $\lambda$,  whence
				\begin{align*}
								\left\langle \frac{\partial G}{\partial \lambda} (\lambda,\theta), n(\theta)\right\rangle > 0.
				\end{align*}
				We deduce that for any fixed $\theta$, in the basis $(n(\theta), \tau(\theta))$, the Jacobian of $G$, denoted $J_G$,  is triangular with positive diagonal entries. More precisely fix $\lambda, \theta$ and set $x = G(\lambda, \theta)$ such that $\lambda = f(x)$, $\theta = \theta(x)$. In the basis $(n(\theta), \tau(\theta))$, we deduce from \eqref{eq:derivativeGtheta} that the Jacobian of $G$ is of the form
				\begin{align*}
				    J_G(\lambda, \theta) = 
				    \begin{pmatrix}
				        \alpha &0\\
				        \gamma&\beta
				    \end{pmatrix},
				\end{align*}
				where  
				\begin{align*}
				        \alpha &= \left\langle \frac{\partial G}{\partial \lambda} (\lambda,\theta), n(\theta)\right\rangle>0  \\
				        \beta &= \left\langle \frac{\partial G}{\partial \theta} (\lambda,\theta), \tau(\theta)\right\rangle = a(\lambda) \rho_-(\theta) + b(\lambda) \rho_+ (\theta)   >0\\
				        \gamma &=  \left\langle \frac{\partial G}{\partial \lambda} (\lambda,\theta), \tau(\theta)\right\rangle.
				\end{align*}
				It is thus invertible and we have a local diffeomorphism. We deduce that 
				\begin{align*}
				    J_G(\lambda, \theta)^{-1} = 
				    \begin{pmatrix}
				        \alpha^{-1} &0\\
				        \frac{-\gamma}{\alpha\beta}&\beta^{-1}
				    \end{pmatrix}.
				\end{align*}
			    We have $J_G(\lambda,\theta)^{-1} = J_{G^{-1}}(x)$ so that the first line is $\nabla f(x)$ and second line is $\nabla \theta(x)$, which proves the claimed expressions for gradients. 
			    
			    We also have $d n(\theta) / d\theta = \tau(\theta)$ so that 
				\begin{align*}
				    J_{n \circ \theta}(x) = \tau(\theta(x)) \nabla \theta(x)^T.
				\end{align*}
			    Differentiating the gradient expression, we obtain ($\nabla$ denotes gradient with respect to $x$):
			    \begin{align*}
			        &\nabla^2 f(x)\\
			        =\quad& \frac{\partial}{\partial  x} \left(\frac{1}{\left\langle \frac{\partial G}{\partial \lambda} (f(x), \theta(x)), n(\theta(x)) \right\rangle} \, n(\theta(x))\right)\\
			        =\quad& \frac{1}{\left\langle \frac{\partial G}{\partial \lambda} (f(x), \theta(x)), n(\theta(x)) \right\rangle} \, J_{n\circ \theta}(x) \\
			        &+ n(\theta(x)) \nabla \left(\frac{1}{\left\langle \frac{\partial G}{\partial \lambda} (f(x), \theta(x)), n(\theta(x)) \right\rangle }\right)^T  \\
			        =\quad& \frac{1}{\left\langle \frac{\partial G}{\partial \lambda} (f(x), \theta(x)), n(\theta(x)) \right\rangle} \, J_{n\circ \theta}(x) \\
			        &- n(\theta(x)) \nabla \left(\left\langle \frac{\partial G}{\partial \lambda} (f(x), \theta(x)), n(\theta(x)) \right\rangle \right)^T \frac{1}{\left\langle \frac{\partial G}{\partial \lambda} (f(x), \theta(x)), n(\theta(x)) \right\rangle^2} \\
			        =\quad & \frac{1}{\left\langle \frac{\partial G}{\partial \lambda} (f(x), \theta(x)), n(\theta(x)) \right\rangle} \left( \tau(\theta) \nabla \theta(x)^T - \nabla f(x)  \nabla \left(\left\langle \frac{\partial G}{\partial \lambda} (f(x), \theta(x)), n(\theta(x)) \right\rangle \right)^T\right).
			    \end{align*}
			    We have
			    \begin{align*}
			         =\quad&\nabla \left(\left\langle \frac{\partial G}{\partial \lambda} (f(x), \theta(x)), n(\theta(x)) \right\rangle \right)^T \\
			         =\quad & \frac{\partial G}{\partial \lambda} (f(x), \theta(x))^T J_{n\circ\theta}(x) + n(\theta(x))^T J_{\frac{\partial G}{\partial \lambda}}(f(x),\theta(x)) J_{G^{-1}} (x)  \\
			         =\quad & \left\langle \frac{\partial G}{\partial \lambda}(f(x), \theta(x)), \tau(\theta(x)) \right\rangle \nabla \theta(x)^T + \left\langle n (\theta(x)), \frac{\partial^2 G}{\partial \lambda^2}(\lambda,\theta)\right\rangle \nabla f(x)^T,
			    \end{align*}
			    where, for the last identity, we have used the fact that
			    \begin{align*}
			        n(\theta(x))^T J_{\frac{\partial G}{\partial \lambda}}(f(x),\theta(x)) &= \left\langle n (\theta(x)), \frac{\partial^2 G}{\partial \lambda^2}(\lambda,\theta)\right\rangle n(\theta(x))^T + \left\langle n (\theta(x)), \frac{\partial^2 G}{\partial \lambda \partial \theta}(\lambda,\theta)\right\rangle \tau(\theta(x))^T\\
			        &= \left\langle n (\theta(x)), \frac{\partial^2 G}{\partial \lambda^2}(\lambda,\theta)\right\rangle n(\theta(x))^T\\
			        n(\theta(x))^T J_{G^{-1}}(x) &= n(\theta(x))^T  J_G(f(x),\theta(x))^{-1}\\
			        &= \frac{1}{\left\langle \frac{\partial G}{\partial \lambda} (f(x),\theta(x)), n(\theta(x))\right\rangle} n(\theta(x))^T = \nabla f(x)^T.
			    \end{align*}
			    We deduce that 
			    \begin{align*}
			        &\nabla^2 f(x)\\
			        =\quad & \frac{1}{\left\langle \frac{\partial G}{\partial \lambda} (f(x), \theta(x)), n(\theta(x)) \right\rangle}\\
			        &\Bigg( \tau(\theta(x)) \nabla \theta(x)^T - \left\langle \frac{\partial G}{\partial \lambda}(f(x), \theta(x)), \tau(\theta(x)) \right\rangle \nabla f(x) \nabla \theta(x)^T  \\
			        &-  \nabla f(x) \nabla f(x)^T \left\langle n (\theta(x)), \frac{\partial^2 G}{\partial \lambda^2}(\lambda,\theta)\right\rangle \Bigg).
			    \end{align*}
			    We have 
			    \begin{align*}
			       & \tau(\theta(x)) \nabla \theta(x)^T - \left\langle \frac{\partial G}{\partial \lambda}(f(x), \theta(x)), \tau(\theta(x)) \right\rangle \nabla f(x) \nabla \theta(x)^T \\
			        =\quad& \left(\tau(\theta(x)) - \frac{\left\langle \frac{\partial G}{\partial \lambda}(f(x), \theta(x)), \tau(\theta(x)) \right\rangle}{\left\langle \frac{\partial G}{\partial \lambda}(f(x), \theta(x)), n(\theta(x)) \right\rangle} n(\theta(x)) \right) \nabla \theta(x)^T\\
			        =\quad& (a(\lambda) \rho_-(\theta) + b(\lambda) \rho_+ (\theta)) \nabla \theta(x) \nabla \theta (x)^T
			    \end{align*}
			    So that we actually get
			    \begin{align*}
			        &\nabla^2 f(x) \left\langle \frac{\partial G}{\partial \lambda} (f(x), \theta(x)), n(\theta(x)) \right\rangle\\
			        =\quad &
			        (a(\lambda) \rho_-(\theta) + b(\lambda) \rho_+ (\theta)) \nabla \theta(x) \nabla \theta (x)^T -   \left\langle n (\theta(x)), \frac{\partial^2 G}{\partial \lambda^2}(\lambda,\theta)\right\rangle \nabla f(x) \nabla f(x)^T\\
			        =\quad &\left\langle\frac{\partial G}{\partial \theta}(\lambda,\theta),\tau(\theta(x))\right\rangle
			         \nabla \theta(x) \nabla \theta (x)^T -   \left\langle n (\theta(x)), \frac{\partial^2 G}{\partial \lambda^2}(\lambda,\theta)\right\rangle \nabla f(x) \nabla f(x)^T.
			    \end{align*}
			    This concludes the proof.
\end{proof}

\subsection{Smooth convex interpolation of smooth positively curved convex sequences}

In this section, we consider an indexing set $I$ with either $I = \NN$ or $I = \ZZ$, and an increasing sequence of compact convex sets $\left( T_i \right)_{i\in I}$ such that for any $i $ in $ I$, the couple  $T_+ := T_{i+1}$, $T_-: = T_i$ satisfies Assumption~\ref{ass:curvature}. In particular, for each $i $ in $ I$, $T_i$ is compact, convex with $C^k$ boundary and positive  curvature.  We denote by $c_i$ the corresponding parametrization by the normal, $T_i \subset \mathrm{int}\, T_{i+1}$. With no loss of generality we assume  $\displaystyle 0 \in \cap_{i \in I} T_i$. 

This is our main theoretical result.
\begin{theorem}[Smooth convex interpolation]
				\label{th:smoothinterp}
        Let $I = \NN$ or $I = \ZZ$ and $\left( T_i \right)_{i\in I}$ such that for any $i \in I$, $T_i \subset \RR^2$ and the couple  $T_+ := T_{i+1}$, $T_-: = T_i$ satisfies Assumption \ref{ass:curvature}.
                Then there exists a $C^k$ convex function 
				\begin{align*}
								f \colon \T:= \mathrm{int}\left(\bigcup_{i\in I} T_i   \right) \mapsto \RR
				\end{align*}
				such that 
				
				(i) $T_i$ is a sublevel set of $f$ for all $i $ in $ I$.

				(ii) We have 
				\begin{align*}
				    \argmin f= \begin{cases}
				    \bigcap_I T_i &\quad \text{ if } I = \ZZ \\
				    \{0\}&\quad \text{ if } I = \NN.
				\end{cases}
				\end{align*}
				
				(iii) $\nabla^2 f $ is positive definite on $\T \setminus \argmin f$, and if $I = \NN$, it is positive definite throughout $\T$. 	
				\end{theorem}
\begin{proof}$\:$\\ 
\noindent
{\em Preconditionning.}	We have that $0 \in \cap_{i \in I } \mathrm{int}(T_i)$. Hence for any $i $ in $ I$ and $0 \leq \alpha <1$, $\alpha T_i \in \mathrm{int}(T_i)$. Furthermore, for $\alpha>0$, small enough, $(1 + \alpha)T_{i} \subset (1-\alpha) \mathrm{int}(T_{i+1})$. 
Set $\alpha_0$ such that $(1 + \alpha_0)T_0 \subset (1-\alpha_0) \mathrm{int} (T_1)$.  By forward (and backward if $I = \ZZ$) induction, for all $i  $ in $ I$, we obtain $\alpha_i>0$ such that
				\begin{align*}
								(1 + \alpha_i) T_i &\subset (1-\alpha_{i}) \mathrm{int} (T_{i+1}).
				\end{align*}
				Setting for all $i $ in $ I$, $\epsilon_{i+1} = \min\{\alpha_i,\alpha_{i+1}\}$ ($\epsilon_0 = \alpha_0$ if $I = \NN$), we have for all $i $ in $ I$
				\begin{align*}
				    (1 + \epsilon_i) T_i \subset (1 + \alpha_i) T_i &\subset (1-\alpha_{i}) \mathrm{int} (T_{i+1}) \subset(1-\epsilon_{i+1}) \mathrm{int} (T_{i+1}).
				\end{align*}
For all $i $ in $ I$, we introduce
\begin{align*}
				S_{3i} &= T_i,\\
				S_{3i + 1} &= (1 + \epsilon_i) T_i\\
				S_{3i + 2} &= (1 - \epsilon_{i+1}) T_{i+1}.
\end{align*} 

We have a new sequence of strictly increasing compact convex sets $\left( S_i\right)_{i \in I }$. 

\smallskip
\noindent
{\em Value assignation.} For each $i $ in $ I $, we set
\begin{align*}
				K_i = \max_{\|x\| = 1} \frac{\sigma_{S_{i+1}}(x)  - \sigma_{S_i}(x)}{ \sigma_{S_{i}}(x)  - \sigma_{S_{i-1}}(x)} \in \left( 0, + \infty \right).
\end{align*}
Note that for all $i $ in $ I$, $K_{3i} = 1$.  We choose $\lambda_{1} = 2$, $\lambda_0 = 1$ and for all $i $ in $ I $, 
\begin{equation}\label{value}
				\lambda_{i+1} = \lambda_i + K_i(\lambda_i - \lambda_{i-1}).
\end{equation}
By construction, we have for all $i $ in $ I $ and all $\theta \in \RR / 2 \pi \Z $, 
\begin{align*}
				\frac{\sigma_{S_{i+1}}(n(\theta))   - \sigma_{S_{i}}(n(\theta))}{\lambda_{i+1} - \lambda_{i}} \leq	\frac{\sigma_{S_{i}}(n(\theta)) - \sigma_{S_{i-1}}(n(\theta))  }{\lambda_{i} - \lambda_{i-1}}.
\end{align*}
If $I = \ZZ$, this entails
\begin{align*}
				0 < \lambda_i - \lambda_{i - 1} \leq \frac{\lambda_{1}- \lambda_0}{\sigma_{S_{1}}(n(\theta)) - \sigma_{S_{0}}(n(\theta))} (\sigma_{S_{i}}(n(\theta)) - \sigma_{S_{i-1}}(n(\theta))),
\end{align*}
and the right-hand side is summable over negative indices $i\leq 0$, so that $\lambda_{i} \to \underline{\lambda} \in \RR$ as $i \to -\infty$. In all cases $(\lambda_i)_{i\in I}$ is an increasing sequence bounded from below.

\smallskip
\noindent
{\em Local interpolation.} We fix $i $ in $ I $ and consider the function $G_i$ described in Proposition~\ref{prop:diffeoMorphism} with $T_+ = S_{3i + 3} = T_{i+1}$, $T_- = S_{3i} = T_i$, $\lambda_+ = \lambda_{3i+3}$, $\lambda_{1} = \lambda_{3i+2}$, $\lambda_0 = \lambda_{3i+1}$, $\lambda_- = \lambda_{3i}$, $e_0 = \epsilon_i$, $e_1 = \epsilon_{i+1}$. By linearity, we have for any $(\lambda,\theta) \in [\lambda_-,\lambda_+] \times \RR / 2\pi \Z $,
\begin{align*}
				\left\langle G_i(\lambda,\theta),n(\theta)\right\rangle = \tilde \gamma_{\Theta}(\lambda)
\end{align*}
where $ \tilde \gamma_{\Theta}$ is as in Lemma~\ref{lem:convexInterpol} with input data $q_0 = \left\langle c_{3i}(\theta), n(\theta)\right\rangle = \sigma_{S_{3i}}(n(\theta))$, $q_1 = \left\langle c_{3i+3}(\theta), n(\theta)\right\rangle = \sigma_{S_{3i+3}}(n(\theta))$, and  $\lambda_-,\lambda_0,\lambda_1,\lambda_+,e_1,e_0$ as already described. This corresponds to the Bernstein approximation of the piecewise affine interpolation between the points
\begin{align}
                &\left( \lambda_{3i}, \sigma_{S_{3i}}(n(\theta))\right)\nonumber\\
                &\left( \lambda_{3i+1}, \sigma_{S_{3i+1}}(n(\theta))\right),\nonumber \\
                &\left( \lambda_{3i+2}, \sigma_{S_{3i+2}}(n(\theta))\right),\nonumber \\
				&(\lambda_{3i+3}, \sigma_{3i+3}(n(\theta))),	\label{eq:piecewiseAffineInterpolExplicit}
\end{align}
By construction of $\left( K_i \right)_{i \in I }$, we have for all $\theta$,
\begin{align*}
				0 < \frac{\sigma_{S_{3i+3}}(n(\theta)) - \sigma_{S_{3i+2}}(n(\theta))  }{\lambda_{3i+3} - \lambda_{3i + 2}} \leq	\frac{\sigma_{S_{3i+2}}(n(\theta)) - \sigma_{S_{3i+1}}(n(\theta))  }{\lambda_{3i+2} - \lambda_{3i + 1}} \leq \frac{\sigma_{S_{3i+1}}(n(\theta)) - \sigma_{S_{3i}}(n(\theta))  }{\lambda_{3i+1} - \lambda_{3i}}.
\end{align*}
Whence the affine interpolant  between points in \eqref{eq:piecewiseAffineInterpolExplicit} is strictly increasing and concave, and by using the shape preserving properties of Bernstein polynomials, $\left\langle G_i(\lambda,\theta),n(\theta)\right\rangle$ has strictly positive derivative. As a consequence $G_i$ is a diffeomorphism and its derivatives are as  in Proposition \ref{prop:diffeoMorphism}. Furthermore
\begin{align*}
				\lambda \mapsto \left\langle G_i(\lambda,\theta),n(\theta)\right\rangle  
\end{align*}
is a $C^k$ concave function of $\lambda$.

\smallskip
\noindent
{\em Global interpolation.} Recall that $\underline{\lambda} = \inf_{i \in I} \lambda_i>-\infty$ and set  $\bar{\lambda} = \sup_{i \in I} \lambda_i\in (-\infty,+\infty]$. For any $\lambda \in (\underline{\lambda},\bar{\lambda})$, there exists a unique $i_\lambda \in I$ such that $\lambda \in [\lambda_{3i_\lambda}, \lambda_{3i_\lambda+3})$. Define 
\begin{align*}
    G \colon (\underline{\lambda}, \bar{\lambda}) \times \RR / 2\pi \Z &\mapsto \RR^2 \\
    (\lambda,\theta) &\mapsto G_{i_\lambda} (\lambda, \theta).
\end{align*}
Fix $i$ in $I$. The boundary of $T_{i+1}$ is given by $G_{i+1}(\lambda_{3i+3}, \RR / 2 \pi \Z ) = G_{i}(\lambda_{3i+3}, \RR / 2 \pi \Z )$ with actually
\begin{equation}\label{coinc}
G_{i+1}(\lambda_{3i+3}, \theta) = G_{i}(\lambda_{3i+3}, \theta ), \mbox{ for all }\theta \mbox{ in }\RR / 2 \pi \Z.\end{equation}

Since $K_{3i} = 1$, we have
\begin{align*}
				\lambda_{3i+1} - \lambda_{3i} = \lambda_{3i} - \lambda_{3i-1}.
\end{align*}
The expressions of the derivatives in Proposition \ref{prop:diffeoMorphism} and \eqref{coinc} ensure that the derivatives of $G_{i+1}$ and $G_{i}$ agree on ${\lambda_{3i+3}} \times \RR / 2 \pi \Z $ up to order $k$. Hence $G$ is a local diffeomorphism. Bijectivity of each $G_i$ ensure that $G$ is also bijective and thus $G$ is a diffeomorphism.
 Furthermore
\begin{align*}
				\lambda \mapsto \left\langle G(\lambda,\theta),n(\theta)\right\rangle  
\end{align*}
is  $C^k$ piecewise concave and thus concave. 

\smallskip
\noindent
{\em Extending $G$.} If $I = \NN$, we may assume without loss of generality that $S_0 = B$ the Euclidean ball and $S_1 = 5/3 S_0$, which corresponds to $\epsilon_0 = 2/3$, eventually after adding a set in the list and rescaling. Let $\phi$ denote the function described in Lemma \ref{lem:interpolationAroundZero} and $G_{-1}$ be described as in Lemma \ref{lem:diffGauge2}. This allows to extend $G$ for $\lambda \in [0,1]$, $G$ is then $C^k$ on $(0,\bar{\lambda}) \times\RR / 2\pi \ZZ$ by using Lemma \ref{lem:diffGauge2} and Proposition \ref{prop:diffeoMorphism}. This does not affect the differentiability, monotonicity and concavity properties of $G$. 

\smallskip
\noindent
{\em Defining the interpolant $f$.} We assume without loss of generality that $\underline{\lambda} = 0$. We set $f$ to be the first component of the inverse of $G$ so that it is defined on $G^{-1}\left( (0,\bar{\lambda}) \times \RR / 2\pi \Z \right)$. We extend $f$ as follows:
\begin{itemize}
				\item $f(0) = 0$ if $I = \NN$,
				\item $f = 0$ on $\cap_{i \in I } T_i$ if $I = \ZZ$.
\end{itemize}
Since $G$ is $C^k$ and non-singular on $ (0,\bar{\lambda}) \times \RR / 2\pi \Z$, the inverse mapping theorem ensures that $f$ is $C^k$ on $\mathrm{int}(\T) \setminus \arg\min_{\T} f$.

\smallskip
\noindent
{\em Convexity of $f$.} For any $\theta $ in $ \RR/ 2\pi \Z $,
\begin{align*}
(0, \bar{\lambda}) &\mapsto \RR_+\\
				\lambda &\mapsto \sup_{z\in[f \leq \lambda]} n(\theta)^Tz 
\end{align*}
is equal to $\left\langle G(\lambda,\theta),n(\theta)\right\rangle$ which is concave. It can be extended at  $\lambda = 0$ by continuity. This preserves concavity hence, using Theorem~\ref{th:crouzeix}, we have proved that $f$ is convex and $C^k$ on $\T\setminus \argmin_{\T}f$.

\smallskip
\noindent
{\em Smoothness around the argmin and Hessian positivity.} If $I = \NN$, then the interpolant defined in Lemma \ref{lem:diffGauge2} ensures that $f$ is proportional to the norm squared around $0$. Hence it is $C^k$ around $0$ with positive definite Hessian. We may compose $f$ with the function $g \colon t \mapsto \sqrt{t^2 + 1} + t$ which is increasing and has positive second derivative. This ensures that the resulting Hessian is positive definite outside $\argmin f$ and thus everywhere since
\begin{align*}
				\nabla^2 g \circ f = g' \nabla^2 f + g'' \nabla f \nabla f^T
\end{align*}
 is positive definite thanks to the expressions for the Hessian of $f$ in Proposition \ref{prop:diffeoMorphism} 

If $I = \ZZ$, we let all the derivatives of $f$  vanish around the solution set. The smoothing Lemma \ref{lem:CkSmoothing} applies and provides a function $\phi$ with positive derivative on $(0,+\infty)$, such that $\phi \circ f$ is convex,  $C^k$ with prescribed sublevel sets. Furthermore, we remark that 
\begin{align*}
				\nabla^2 \phi \circ f = \phi' \nabla^2 f + \phi'' \nabla f \nabla f^T.
\end{align*}
We may compose $\phi \circ f$ with the function $g \colon t \mapsto \sqrt{t^2 + 1} + t$ which is increasing with positive second derivative, the expressions for the Hessian of $f$ in Proposition \ref{prop:diffeoMorphism} ensure that the resulting Hessian is positive definite out of $\argmin f$.\end{proof}

\begin{remark}[Interpolation of symmetric rings]\label{rem:alignedLevelSets}{\rm In view of Remark~\ref{rem:alignedInterpolation}, if we have $T_{i+1} = \alpha T_i$ for some $0<\alpha < 1$ and $i $ in $ \Z $, then the interpolated level sets between $T_i$ and $T_{i+1}$ are all of the form $s T_i$ for $\alpha \leq s \leq 1$.}
\end{remark}

\begin{remark}[Strict convexity]\label{rem:strictConvexity}{\rm 
				Recall that strict convexity of a differentiable function amounts to the injectivity of its gradient. In Theorem \ref{th:smoothinterp} if there is a unique minimizer, then the  invertibility of the Hessian outside  argmin~$f$ ensures that our interpolant is strictly convex (note that this is automatically the case if $I = \NN$).}
\end{remark}

\subsection{Considerations on Legendre functions}\label{s:legendre}

The following proposition provides some interpolant with additional properties as global Lipschitz continuity and finiteness properties for the dual function. At this stage these results appear as 
merely technical but they happen 
to be decisive in the 
construction of counterexamples 
involving  Legendre 
functions. The properties of Legendre functions can be found in  \cite[Chapter 6]{rockafellar1970convex}. We simply recall here that, given a convex body $C$ of $\R^p$, a convex function $h:C\to \R$ is {\em Legendre} if it is differentiable on $\inte C$ and if $\nabla h$ defines a bijection from $\inte C$ to $\nabla h(\inte C)$ with in addition 
$$\lim_{\begin{array}{l}
x\in \inte C\\
x\to z\end{array}} \|\nabla h(x)\|=+\infty,$$
for all $z$ in $\bd C$. We also assume that $\mbox{epi}\,f:=\{(x,\lambda):f(x)\leq \lambda\}$ is closed in $\R^{p+1}$. The {\em Legendre conjugate} or {\em dual function} of $h$ is defined through
$$h^*(z)=\sup\left\{\langle z,x\rangle -h(x):x \in C\right\},$$
for $z$ in $\R^p$, and its domain is $D:=\left\{z\in\R^p:h(z)<+\infty\right\}.$ The function $h^*$ is differentiable on the interior of $D$,  and the inverse of $\nabla h:\inte C \to \inte D$ is $\nabla h^*:\inte D \to \inte C$.

We start with a simple technical lemma on the compactness of the domain of a Legendre function.
\begin{lemma}
				Let $h \colon \RR^2 \mapsto \RR$ be a globally Lipschitz continuous Legendre function, and set $D = \mathrm{int}(\mathrm{dom}(h^*))$ where $h^* \colon \RR^2 \mapsto \RR$ is the conjugate of $h$. For each $\lambda \geq \min_{\RR^2}h$ let  $\sigma_\lambda$ be the support function associated to the set $\left\{ z \in \RR^2, \, h(z) \leq \lambda \right\}$. The following are equivalent
				\begin{itemize}
								\item[(i)] $h^*(x)\leq 0$ for all $x \in D$.
								\item[(ii)] For all $y \in \RR^2$, $\sigma_{h(y)}(\nabla h(y)) \leq h(y)$.
				\end{itemize}
				In both cases $h^*$ has compact domain.
				\label{lem:legendreBounded}
\end{lemma}
\begin{proof} Let us establish beforehand the following formula 
\begin{align}\label{e:h}
								h^*(z) &= \sigma_{h(y)}(\nabla h(y))- h(y),
				\end{align}
				with $y = \nabla h^* (z)$ and $y\in\R^2$. Since	$h^*$ is Legendre, we have for all $y$ in $D$,
				\begin{align*}
								h^*(z) &= \sup_{y \in \RR^2} \left\langle z, y  \right\rangle - h(y) =  \left\langle z, \nabla h^*(z) \right\rangle - h(\nabla h^*(z)).
				\end{align*}
				We have, setting $y = \nabla h^*(z)$
				\begin{align*}
								\left\langle z, \nabla h^*(z) \right\rangle  = \left\langle \nabla h(y),y \right\rangle  = \sigma_{h(y)}(\nabla h(y))
				\end{align*}
				because $\nabla h(y)$ is normal to the sublevel set of $h$ which contains $y$ in its boundary. Hence we have $h^*(z) = \sigma_{h(y)}(\nabla h(y))- h(y)$
				with $y = \nabla h^* (z)$, that is \eqref{e:h} holds. Since $\nabla h^* \colon D \mapsto \RR^2$ is a bijection, the equivalence follows. In this case the domain of $h^*$ is closed because $h^*$ is bounded and lower semicontinuous. The domain  is also bounded by the Lipschitz continuity of $h$, whence compact. 
\end{proof}

\begin{proposition}[On Legendre interpolation]
        \label{th:globallyLipshitz}
				Let $\left( S_i \right)_{i\in \NN}$ be such that for any $i $ in $ I$, $T_- = S_{i}$,  $T_+ = S_{i+1}$ satisfy Assumption~\ref{ass:curvature} and there exists a sequence $\left( \epsilon_i \right)_{i \in \NN}$ in $(0,1)$ such that for all $i\geq 1$, $(1 - \epsilon_i)^{-1} S_{3i-1} = (1 +\epsilon_i)^{-1} S_{3i+1}=S_{3i}$.\\ Assume in addition that,
				\begin{align}
								\inf_{\|x\| = 1}& \sigma_{S_i}(x) - \sigma_{S_{i-1}}(x) =1 +  O\left(\frac{1}{i^3}\right)&\mbox{ (non degeneracy)}\label{nd}, \\
								\sup_{\|x\|=1}& \left|\frac{\sigma_{S_{i+1}}(x) - \sigma_{S_i}(x)}{\sigma_{S_{i}}(x) - \sigma_{S_{i-1}}(x)} - 1\right| = O\left( \frac{1}{i^3} \right)& \mbox{ (moderate growth)}\label{mg}.
				\end{align}
				Then there exists a convex $C^k$  function $h \colon \RR^2 \mapsto \RR$, such that 
				\begin{itemize}
								\item For all $i $ in $ \NN$, $S_{3i}$ is a sublevel set of $h$,
								\item $h$ has positive definite Hessian,
								\item $h$ is globally Lipschitz continuous,
								\item $h^*$ has a compact domain $D$ and is $C^k$ and strictly convex on $\mathrm{int}(D)$. 
				\end{itemize}
\end{proposition}
\begin{proof}
				The construction of $h$ follows  the exact same principle as that of   Theorem~\ref{th:smoothinterp}. This ensures that the first two points are valid. Note that equation \eqref{nd} implies that the sets sequence grows by at least a fixed amount in each direction as $i$ grows. Hence we have $\T = \RR^2$. \\
				\textit{Global Lipschitz continuity of $h$:}
			    The values of $h$ are defined through 
	            \begin{align*}
								\lambda_{i+1}-\lambda_{i}&=K_i(\lambda_{i}-\lambda_{i-1}),\;\forall i\in \NN^*\\
								K_i &= \max_{\|x\| = 1} \frac{\sigma_{S_{i+1}}(x)  - \sigma_{S_i}(x)}{ \sigma_{S_{i}}(x)  - \sigma_{S_{i-1}}(x)} \in \left( 0, + \infty \right),
				\end{align*}
				so that $K_i = 1 + O(1/i^3)$ thanks to equation \eqref{mg}. 
				Note that the moderate growth assumption entails  
                \begin{align}
                        \sup_{i \in \NN} \sigma_{S_{i+1}}(x) - \sigma_{S_{i}}(x) = O\left( \prod_{i \in \NN^*} K_i\right) = O(1). 
                        \label{eq:suportDiffBounded}
                \end{align}
				For $i\geq 1$, one has
\begin{equation}\label{for}
    \lambda_{i+1}-\lambda_{i}=\prod_{1 \leq j \leq i} K_j(\lambda_1-\lambda_0).
\end{equation}

On the other hand using the bounds \eqref{mg}, \eqref{nd} and the identity \eqref{for}, there exists a constant $\kappa>0$ such that for all $i \geq 1$,  all $\theta $ in $ \RR / 2 \pi \Z $, 
\begin{align}
				\frac{\sigma_{S_{i+1}}(n(\theta)) - \sigma_{S_{i}}(n(\theta))  }{\lambda_{i+1} - \lambda_{i}} =	\frac{\left(\sigma_{S_{i+1}}(n(\theta)) - \sigma_{S_{i}}(n(\theta))\right)  }{\prod_{j=1}^i K_j (\lambda_{1} - \lambda_{0})} \geq \kappa>0.
				\label{eq:tempDerivLambda}
\end{align}

By the interpolation properties described in Lemma \ref{lem:convexInterpol} the function $\left\langle G(\lambda,\theta),n(\theta)\right\rangle$ constructed in Theorem \ref{th:smoothinterp} has derivative with respect to $\lambda$ greater than $\kappa$. Recalling the expression of the gradient as given in Proposition \ref{prop:diffeoMorphism} (and the concavity of $G$ with respect to $\lambda$), this shows that 
$$\|\nabla h(x)\|\leq \frac{1}{\kappa}$$
for all $x $ in $ \RR^2$, and by the mean value theorem, $h$ is globally Lipschitz continuous on $\RR^2$.

\textit{Properties of the dual function:} $h$ is Legendre, its conjugate $h^*$ is therefore Legendre. From the definiteness of $\nabla^2 h$ and the fact that $\nabla h \colon \RR^2 \mapsto \mathrm{int}(D)$ is a bijection, we deduce  that $h^*$ is  $C^k$ by the inverse mapping theorem. So the only property which we need to establish is that $h^*$ has a compact domain, in other words, using Lemma~\ref{lem:legendreBounded},   it is sufficient to show that $\sup_{x \in \mathrm{int} D} h^*(x) \leq 0$.  

Using the notation of the proof of Theorem \ref{th:smoothinterp}, we will show that it is possible to verify that, for all $\lambda,\theta$ in the domain of $G$ 
\begin{align}\label{eq:toBeCheckedForBoundedness}
\frac{\left\langle n(\theta), G(\lambda,\theta)\right\rangle}{\left\langle \frac{\partial G}{\partial \lambda}(\lambda,\theta), n(\theta) \right\rangle} \leq \lambda.
\end{align}
Equation \eqref{eq:toBeCheckedForBoundedness} is indeed the coordinate form of the characterization given in Lemma \ref{lem:legendreBounded}. 
 Let us observe that
\begin{align}\label{e:int}\frac{\partial}{\partial \lambda} \left( \left\langle n(\theta), G(\lambda,\theta)\right\rangle - \lambda \left\langle \frac{\partial G}{\partial \lambda}(\lambda,\theta), n(\theta) \right\rangle \right) = -\lambda \left\langle \frac{\partial^2 G}{\partial \lambda^2}(\lambda,\theta), n(\theta) \right\rangle,
\end{align}
 and  since $G$ is concave, the right hand side is positive. 
  
 Assume that we have proved that,
 \begin{equation}
 \label{e:lim}
 \lambda \mapsto \sup_\theta -\lambda \left\langle \frac{\partial^2 G}{\partial \lambda^2}(\lambda,\theta), n(\theta) \right\rangle
 \end{equation}
has finite integral as $\lambda \to \infty$. Since the function 
 \begin{equation}
 \label{e:funTheta}
 \theta \mapsto \lambda \left\langle \frac{\partial^2 G}{\partial \lambda^2}(\lambda,\theta), n(\theta) \right\rangle
 \end{equation}
is continuous on $\R/2\pi\ZZ$ for any $\lambda$, Lebesgue dominated convergence theorem would ensure that 
\begin{align*}
    \theta \mapsto \int_{\lambda \geq \lambda_0} -\lambda \left\langle \frac{\partial^2 G}{\partial \lambda^2}(\lambda,\theta), n(\theta) \right\rangle d \lambda
\end{align*}
is continuous in $\theta$, so that: 
\begin{align*}
				 \sup_{\theta} \left[\lim_{\lambda \to \infty} \left\langle n(\theta), G(\lambda,\theta)\right\rangle - \lambda \left\langle \frac{\partial G}{\partial \lambda}(\lambda,\theta), n(\theta) \right\rangle\right]<+\infty.
\end{align*}
Shifting values if necessary, we could assume that this upper bound is equal to zero to obtain equation \eqref{eq:toBeCheckedForBoundedness}. The latter being the condition required in Lemma \ref{lem:legendreBounded}, we would have reached a conclusion.

Let us therefore establish that \eqref{e:int} is integrable over $\R_+$. Recall that $G$ is constructed using the Bernstein interpolation given in Lemma \ref{lem:convexInterpol} between successive values of $\lambda$. As a result, for a fixed $\theta$, the function $\left\langle n(\theta), G(\lambda,\theta)\right\rangle $ is the interpolation of the piecewise affine function interpolating
\begin{align}
                &\left( \lambda_{3i}, \sigma_{S_{3i}}(n(\theta))\right)\nonumber\\
                &\left( \lambda_{3i+1}, \sigma_{S_{3i+1}}(n(\theta))\right),\nonumber \\
                &\left( \lambda_{3i+2}, \sigma_{S_{3i+2}}(n(\theta))\right),\nonumber \\
								&(\lambda_{3i+3}, \sigma_{3i+3}(n(\theta))),
								\label{eq:piecewiseAffineInterpolExplicit2}
\end{align}
as in equation \eqref{eq:piecewiseAffineInterpolExplicit}. This interpolation is concave and increasing.

Assumption \eqref{mg} ensures that $K_j = 1 + O(1/j^3)$. Then
\begin{align*}
				\prod_{j=1}^m K_j = \bar{K} + O(1 / j^{2})
\end{align*}
where $\bar{K}$ is the finite, positive limit of the product (we can for example perform integral series comparison after taking the logarithm).

The recursion on the values writes for all $i\geq1$
\begin{align*}
				\lambda_{i+1} = \lambda_i + K_i(\lambda_i - \lambda_{i-1}),
\end{align*}
so that 
\begin{align*}
				\lambda_{i+1} - \lambda_i= (\lambda_1 - \lambda_0) \prod_{j=1}^i K_i = (\lambda_1 - \lambda_0) \bar{K} + O(1/i^{2}).
\end{align*}

This means that the gap  between consecutive values 
 tends to be constant. Thus by \eqref{e:deg} in Lemma~\ref{lem:convexInterpol}, see also  Remark~\ref{rem:alignedInterpolation}, the degree of the Bernstein  interpolants is bounded. 
Using this bound together with inequality \eqref{e:born}, providing bounds for the derivatives of Bernstein's polynomial,    ensure that, for all $\lambda $ in $ [\lambda_{3i}, \lambda_{3i+3})$:  
\begin{align*}
				&\left|\left\langle \frac{\partial^2 G}{\partial \lambda^2}(\lambda,\theta), n(\theta) \right\rangle\right| \\
				=\,& O\left(  \max _{j = 3i+2, 3i+1}\left|  \frac{\sigma_{S_{j+1}}(n(\theta))   - \sigma_{S_{j}}(n(\theta))}{\lambda_{j+1} - \lambda_{j}} -	\frac{\sigma_{S_{j}}(n(\theta)) - \sigma_{S_{j-1}}(n(\theta))  }{\lambda_{j} - \lambda_{j-1}}  \right| \right).
\end{align*}
Now for any $j = 3i+2, 3i+1$, 
\begin{align*}
    &\left|  \frac{\sigma_{S_{j+1}}(n(\theta))   - \sigma_{S_{j}}(n(\theta))}{\lambda_{j+1} - \lambda_{j}} -	\frac{\sigma_{S_{j}}(n(\theta)) - \sigma_{S_{j-1}}(n(\theta))  }{\lambda_{j} - \lambda_{j-1}}  \right| \\
    =\,& \frac{1}{\lambda_{j+1} - \lambda_{j}} \left|  (\sigma_{S_{j+1}}(n(\theta))   - \sigma_{S_{j}}(n(\theta))) -	\frac{\lambda_{j+1} - \lambda_{j}}{\lambda_{j} - \lambda_{j-1}} (\sigma_{S_{j}}(n(\theta)) - \sigma_{S_{j-1}(n(\theta))  }) \right|\\
    =\,& \left(1 / ((\lambda_1 - \lambda_0) \bar{K}) + O(1/i^2)\right)\\
    &\times \left|  (\sigma_{S_{j+1}}(n(\theta))   - \sigma_{S_{j}}(n(\theta))) -	(1 + O(1/i^2)) (\sigma_{S_{j}}(n(\theta)) - \sigma_{S_{j-1}(n(\theta))  }) \right|\\
     =\,& \left(1 / ((\lambda_1 - \lambda_0) \bar{K}) + O(1/i^2)\right)\\
    &\times \left|  (\sigma_{S_{j+1}}(n(\theta))   - \sigma_{S_{j}}(n(\theta))) -	 (\sigma_{S_{j}}(n(\theta)) - \sigma_{S_{j-1}(n(\theta))  }) \right|
\end{align*}
where the last identity follows from the triangle inequality because using $\sigma_{S_{j}}(n(\theta)) - \sigma_{S_{j-1}}(n(\theta)) = O(1)$ in \eqref{eq:suportDiffBounded}.
Hence
\begin{align*}
                &\left|\left\langle \frac{\partial^2 G}{\partial \lambda^2}(\lambda,\theta), n(\theta) \right\rangle\right| \\
				=\,& 
				\left(1 / ((\lambda_1 - \lambda_0) \bar{K}) + O(1/i^2)\right)\\
				&\times O\left( \max _{j = 3i+2, 3i+1}\left|  \sigma_{S_{j+1}}(n(\theta))   - \sigma_{S_{j}}(n(\theta)) -	(\sigma_{S_{j}}(n(\theta)) - \sigma_{S_{j-1}}(n(\theta)) ) \right| \right)\\
				=\,& 
				\left(1 / ((\lambda_1 - \lambda_0) \bar{K}) + O(1/i^2)\right)\\
				&\times O\left(  \max _{j = 3i+2, 3i+1}\left|\sigma_{S_{j}}(n(\theta)) - \sigma_{S_{j-1}}(n(\theta))  \right| \times \left|  \frac{\sigma_{S_{j+1}}(n(\theta))   - \sigma_{S_{j}}(n(\theta))}{\sigma_{S_{j}}(n(\theta)) - \sigma_{S_{j-1}}(n(\theta)) } -	1 \right| \right)\\
				=\,& O(1/i^3),
\end{align*}
where the last inequality follows from \eqref{eq:suportDiffBounded} and \eqref{mg}. 
Now as $i \to \infty$, $\lambda_{3i} \sim \lambda_{3i + 3} \sim i c$ for some constant $c>0$ and
\begin{align*}
				\sup_{\lambda \in [\lambda_{3i}, \lambda_{3i+3}], \theta \in [0, 2 \pi]} - \lambda \left\langle \frac{\partial^2 G}{\partial \lambda^2}(\lambda,\theta), n(\theta) \right\rangle = O(1 / i^2)
\end{align*}
and 
\begin{align*}
				\sup_{ \theta \in [0, 2 \pi]} - \lambda \left\langle \frac{\partial^2 G}{\partial \lambda^2}(\lambda,\theta), n(\theta) \right\rangle
\end{align*}
has finite integral as $\lambda \to \infty$. This implies \eqref{e:lim} and it  concludes the proofs.
 \end{proof}

\section{Smooth convex interpolation for sequences of polygons}

Given a sequence of points  $A_1,\ldots,A_n$, we denote by $A_1\ldots A_n$ the polygon obtained by joining successive points ending the loop with the segment $[A_n,A_1]$. In the sequel we consider mainly convex polygons, so that the vertices $A_1,\ldots,A_n$ are also the extreme points.

The purpose of this section is first to show that polygons can be approximated by smooth convex sets with prescribed normals under weak assumptions. Figure~\ref{fig:illustrPolySmooth} illustrates the result we would like to establish: given a target polygon with 
prescribed normals at its vertices, we wish to construct a smooth convex set interpolating the vertices with the desired normals and whose distance to the polygon is small. 

Then given a sequence of nested polygons, we provide a smooth convex function which interpolates the polygons in the sense described just above.

Given a closed nonempty convex subset $S$ of  $\R^p$  and $x$ in $S$, we recall that the {\em normal cone to $S$ at $x$} is 
$$N_S(x)=\left\{z\in \R^p:\langle z,y-x\rangle \leq 0,\forall y\in S\right\}.$$
Such vectors will often simply called normals (to $S$) at $x$.

\subsection{Smooth approximations of polygons}

\begin{lemma}				\label{lem:approxSegment}
				For any $r_-,r_+ > 0$, $t_- > 0$, $t_+<0$ and $\epsilon > 0$, $m \in \NN$, $m \geq 3$, there exists a strictly concave polynomial function $p \colon [0,1] \mapsto [0,\epsilon]$ such that
				\begin{align*}
								p(0) &= 0 &p(1) &= 0\\
								p'(0) &= t_-& p'(1) &= t_+\\
								p''(0) &= -r_- &p''(1) &= -r_+.\\
								p^{(q)}(0) & = 0 \quad  q \in\{3, \ldots, m\}.
				\end{align*}
\end{lemma}

\begin{figure}
				\centering
				\includegraphics[width=.7\textwidth]{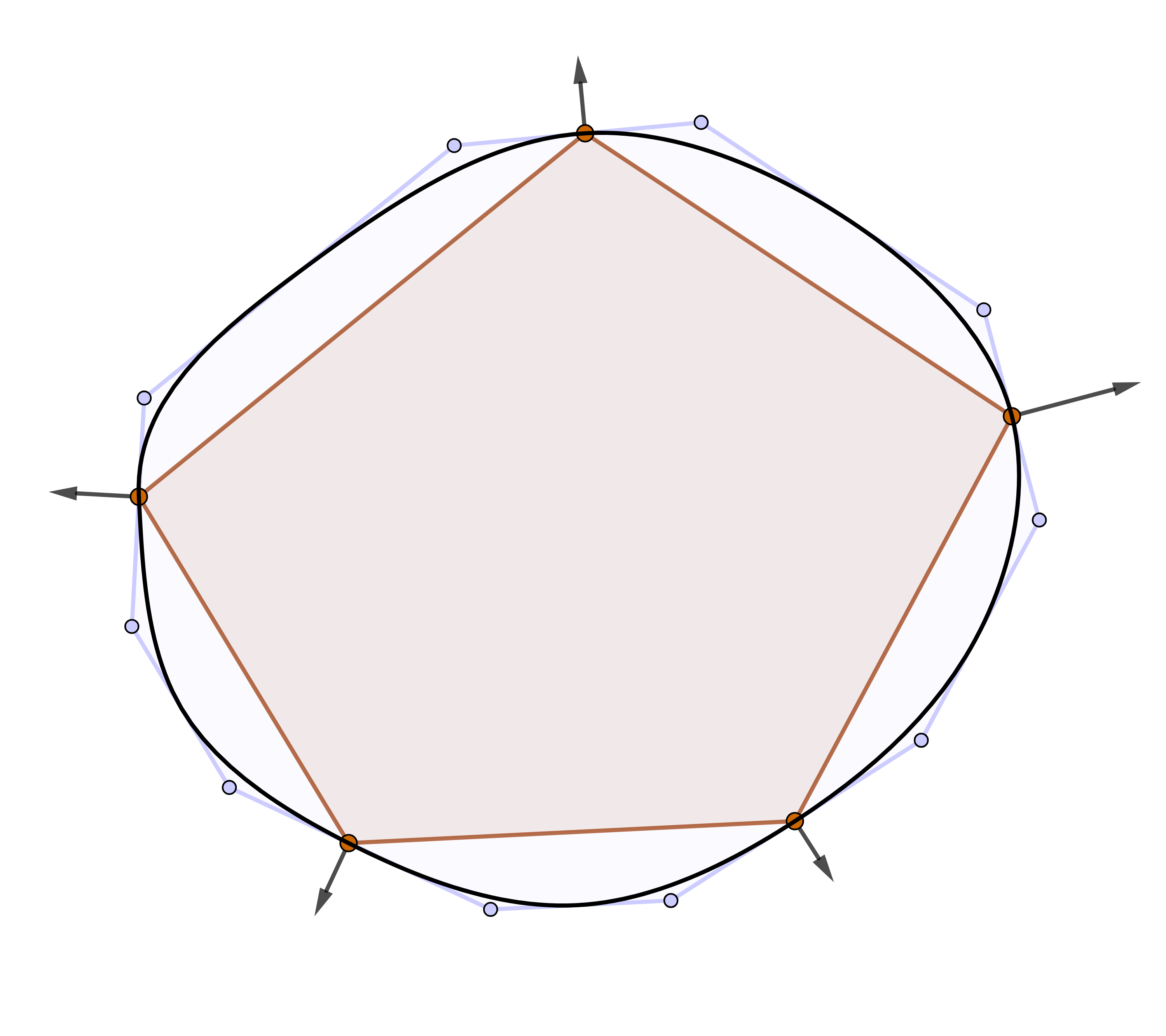}
				\caption{Arrows designate  the prescribed normals. We construct a strictly convex set with smooth boundary entirely contained in the auxiliary (blue) polygon. This set interpolates the normals and the distance to the original (red) polygon can be chosen arbitrarily small. The degree of smoothness of the boundary can be chosen arbitrarily high.}
				\label{fig:illustrPolySmooth}
\end{figure}
\begin{proof} Let us begin with preliminary remarks. 
				Consider for any $a,b $ in $ \RR$, the function 
				\begin{align*}
								f \colon t \mapsto a(t-b)^2.
				\end{align*}
				For any $t $ in $ \RR$, $q $ in $ \NN$, $q > 2$, and any $c > 0$, we have
				\begin{align}
								f(t + c) - f(t) &= \Delta^1_c f(t) = a c ( 2(t-b) + c)\nonumber\\
								f(t + 2 c) - 2 f(t+c) + f(t) &= \Delta^2_c f(t) = a c ( 2 (t + c - b) + c - 2(t-b) - c) = 2ac^2\nonumber\\
								\Delta_c^q f(t) &= 0
								\label{eq:approxSegment00}
				\end{align}
				
				\noindent
				{\em Choosing the degree $d$ and constructing the polynomial. }
				For any $d $ in $ \NN$, $d \geq 2m + 1$, we set
				\begin{align}
								a_-(d) &= \frac{-dr_-}{2(d - 1)}<0&b_-(d) &= \frac{1}{2d} \left( 1 + 2 t_-\frac{d-1}{r_-} \right)> 0\nonumber\\
								a_+(d) &= \frac{-dr_+}{2(d - 1)}<0&b_+(d) &= 1 + \frac{1}{2d} \left( -1 + 2t_+\frac{d-1}{r_+} \right)< 1,
								\label{eq:approxSegment0}
				\end{align}
			 	and define the functions
				\begin{align*}
								f_d \colon s &\mapsto
								\begin{cases}
												a_-(d) (( s - b_-(d))^2 - b_-(d)^2) & \text{ if } s\leq b_-(d)\\
												- a_-(d) b_-(d)^2& \text{ if } s\geq b_-(d)
								\end{cases}\\
								g_d \colon t &\mapsto
								\begin{cases}
												a_+(d) (( s - b_+(d))^2 - (1-b_+(d))^2) & \text{ if } s\geq b_+(d)\\
												- a_+(d) (1-b_+(d))^2& \text{ if } s\leq b_+(d).
								\end{cases}
				\end{align*}
				Furthermore, we set
				\begin{align*}
								f \colon t &\mapsto
								\begin{cases}
												\frac{r_-}{2} \left( \left( \frac{t_-}{r_-}  \right)^2- \left(s - \frac{t_-}{r_-}  \right)^2 \right)  & \text{ if } s\leq \frac{t_-}{r_-}\\
												\frac{r_-}{2} \left( \frac{t_-}{r_-}  \right)^2 & \text{ if } s\geq \frac{t_-}{r_-}
								\end{cases}\\
								g \colon t &\mapsto
								\begin{cases}
												\frac{r_+}{2} \left(\left( \frac{t_+}{r_+}  \right)^2 -  \left(s -1 - \frac{t_+}{r_+}  \right)^2 \right)  & \text{ if } s\geq 1 + \frac{t_+}{r_+}\\
												\frac{r_+}{2} \left( \frac{t_+}{r_+}  \right)^2 & \text{ if } s\leq 1 + \frac{t_+}{r_+}.
								\end{cases}
				\end{align*}
				Note that $b_-(d) \to t_- / r_-$, $b_+(d) \to 1 + t_+ / r_+$, $a_-(d) \to - r_-/2$ and $a_+(d) \to - r_+/2$ as $d \to \infty$ so that $f_d \to f$ and $g_d \to g$ uniformly on $[0,1]$. For any $d$, $f_d$ is concave increasing and $g_d$ is concave decreasing and all of them are Lipschitz continuous on $[0,1]$ with  constants that do not depend on $d$. Note also that $f(0) = 0 < g(0)$ and $g(1) = 0 < f(1)$.
				We choose $d \geq 2m + 1$ such that 
				\begin{align}
								f_d\left( \frac{m}{d} \right) &\leq \min\left( \epsilon, g_d\left( \frac{m}{d} \right) \right)\nonumber\\
								g_d\left(1 -  \frac{m}{d} \right) &\leq \min\left( \epsilon, f_d\left(1 -  \frac{m}{d} \right) \right).
								\label{eq:approxSegment1}
				\end{align}
				Such a $d$ always exists because in both cases, the left hand side converges to $0$ and the right hand side converges to a strictly positive term as $d$ tends to $\infty$.
				For such a $d$, we set $h \colon s \mapsto \min\left\{ f_d(s), g_d(s),\epsilon \right\}$. By construction, $h$ is concave,   agrees with $f_d$ on $\left[ 0,\frac{m}{d} \right] \subset [0,1/2]$ and with $g_d$ on $\left[ 1- \frac{m}{d},1 \right] \subset [1/2,1]$. Using equation \eqref{eq:approxSegment00} with $c = 1/d$, we deduce that
				\begin{align*}
								d (d-1) \Delta^2 h(0) &= d (d-1) \Delta^2 f_d(0) = d(d-1)\frac{2 a_-(d)}{d^2} = -r_-\\
								d \Delta h(0) &= d \Delta f_d(0) =  a_-(d)\left( \frac{1}{d}  -2b_-(d)\right) = t_-\\
								\Delta^{q}h(0) &= 0 = \Delta^{q}f_d(0) \quad \forall m \geq q \geq 3\\
								d (d-1) \Delta^2 h\left(1 - \frac{2}{d}\right) &= d (d-1) \Delta^2 g_d\left(1 - \frac{2}{d}\right)= d(d-1)\frac{2 a_+(d)}{d^2} = -r_+\\
								d \Delta h\left( 1 - \frac{1}{d} \right) &= d \Delta g_d\left( 1 - \frac{1}{d} \right) =  a_+(d) \left( \frac{-1}{d}  + 2 (1 - b_+(d))\right) = t_+\\
								\Delta^{q}h\left( 1 - \frac{m}{d} \right) &= \Delta^{q}g_d\left( 1 - \frac{m}{d} \right) = 0 \quad \forall m \geq q \geq 3.
				\end{align*}
				From the concavity of $h$ and the derivative formula \eqref{eq:bernsteinDerivative}, we deduce that the polynomial $B_{h,d}$ satisfies the desired properties.
\end{proof}
\begin{figure}[H]
				\centering
				\includegraphics[width=\textwidth]{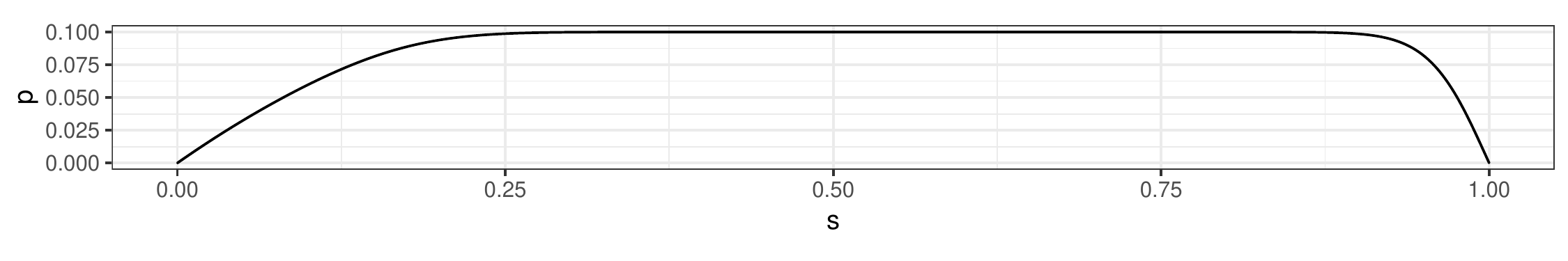}
				\caption{Illustration of the approximation result of Lemma \ref{lem:approxSegment} with $\epsilon = 0.1$, $m = 3$, $t_- =0.7$, $t_+ = -2.2$, $r_- = 2$ and $r_+ = 0.2$. The resulting polynomial is of degree 66. Numerical estimations of the first and second order derivatives at 0 and 1 match the required values up to 3 precision digits. The polynomial is strongly concave, however, this is barely visible because the strong concavity constant is extremely small. }
				\label{fig:illustrApproxSeg}
\end{figure}

We deduce the following result
\begin{lemma}
				Let $a > 0$, $r > 0$, $\epsilon > 0$, and an integer $m \geq 3$. Consider two unit vectors: $v_-$  with strictly positive entries  and $v_+$ with first entry strictly positive and second entry strictly negative. Then there exists a $C^m$ curve $\gamma \colon [0,M] \mapsto \RR^2$, such that 
				\begin{enumerate}
								\item $\|\gamma'\| = 1$.
								\item $\gamma(0) = (-a, 0) := A$ and $\gamma(1) = (0,a) : = B$.
								\item $\gamma'(0) = v_-$ and $\gamma'(1) = v_+$.
								\item $\|\gamma''(0)\| = \|\gamma''(-1)\| = r$.
								\item $\mathrm{det}(\gamma',\gamma'') < 0$ along the curve.
								\item $\gamma^{(q)}(0) = \gamma^{(q)}(1) = 0$ for any $3 \leq q \leq m$.
								\item $\dist(\gamma([0,M]), [A,B]) \leq \epsilon$.
				\end{enumerate}
				\label{lem:existenceCurve}
\end{lemma}
\begin{proof}
				Consider the graph of a polynomial as given in  Lemma \ref{lem:approxSegment} with $t_- = v_-[2] / v_-[1]> 0$, $t_+ = v_+[2] / v_+[1]< 0$ and $r_- = \frac{r}{2a} (1 + t_-^2)^{\frac{3}{2}}$, $r_+ = \frac{r}{2a} (1 + t_+^2)^{\frac{3}{2}}$ and $\epsilon/2a$ as an appro\-xi\-ma\-tion parameter. This graph is parametrized by $t$. It is possible to reparametrize it by arclength to obtain a $C^m$ curve $\gamma_0$ whose tangents at $0$ is $T_-$,  at $1$ is $T_+$, and whose curvature at $0$ and $1$ is $- \frac{r}{2a}$. Furthermore, $\gamma_0$ has strictly negative curvature whence item 5. Consider the affine transform: $x \mapsto 2a(x - 1/2)$, $y \mapsto 2a y$. This results in a $C^m$ curve $\gamma$, parametrized by arclength which satisfies the desired assumptions. 
\end{proof}

\begin{lemma}[Normal approximations of polygons by smooth convex sets]
				Let $S=A_1...A_n$ be a convex polygon. For each $i$,  let  $V_i$ be in $ N_S(A_i)$ such that the angle between $V_i$ and each of the two neighboring faces is within $\left( \frac{\pi}{2},\pi \right)$. Then for any $\epsilon > 0$ and any $m \geq 2$, there exists a compact convex set $C \subset \RR^2$ such that
				\begin{itemize}
								\item[(i)] the boundary of $C$ is $C^m$ with non vanishing curvature,
								\item[(ii)] $S \subset C$,
								\item[(iii)] for any $i = 1,\ldots, n$, $A_i $ in $ \bd(C)$ and the normal cone to $C$ at $A_i$ is given by $V_i$,
								\item[(iv)] $\max_{y \in C} \dist(y,S) \leq \epsilon$.
				\end{itemize}
				\label{lem:polyGon}
\end{lemma}
\begin{proof}
				We assume without loss of generality that $A_1,\ldots,A_n$ are ordered clockwise. For $i = 1, \ldots, n-1$, and each segment $[A_i,A_{i+1}]$, we may perform a rotation and a translation  to obtain $A_i = -(a,0)$ and $A_{i+1} = (a,0)$. Working in this coordinate system, using the angle condition on $V_i$, we may choose $v^-_i$, $v^+_{i+1}$ satisfying the hypotheses of Lemma~\ref{lem:existenceCurve}  respectively orthogonal to $V_i$ and $V_{i+1}$. Choosing $r = 1$, we obtain $\gamma_i \colon [0,M_i] \mapsto \RR^2$ as given by Lemma \ref{lem:existenceCurve}. Rotation and translations affect only the direction of the derivatives of curves, not their length. Hence, it is possible to concatenate curves $\left( \gamma_i \right)_{i=1}^{n-1}$ and to preserve the $C^m$ properties of the resulting curve. At end-points, tangents and second order derivatives coincide while higher derivatives vanish. Furthermore the curvature has constant sign and does not vanish. We obtain a closed $C^m$ curve which defines a convex set which satisfies all the requirements of the lemma.
\end{proof}

\begin{remark}[Bissector]{\rm Given any polygon, choosing normal vectors as given by the direction of the bissector of each angles ensure that the above assumptions are satisfied. Hence all our approximation results hold given polygon without specifying the choice of outer normals.}
\label{rem:noNormal}
\end{remark}

\subsection{Smooth convex interpolants of polygonal sequences}
\begin{definition}[Interpolability]
				{\rm For $n\geq3$, let $A_1\ldots A_n $ be  a convex polygon $S$ and $V_i $ be in $ N_S(V_i)$ for $i=1,\ldots,n$. We say that $\left( A_i,V_i \right)_{i=1}^n $ is  {\em interpolable} if for each $i = 1,\ldots,n$,  the angle between $V_i$ and each of the two neighboring faces of the polygon is in $\left( \frac{\pi}{2},\pi \right)$. 
				The collection $\left( A_i,V_i \right)_{i=1}^n $ is called {\em a polygon-normal pair}.}
				\label{def:interpolable}
\end{definition}
Let $I = \ZZ$ or $I = \NN$. 
Let $\left( PN_i \right)_{i \in I}$ be a sequence of interpolable polygon-normal pairs. Setting for $i $ in $ I$, $PN_i = \left\{ \left( A_{j,i} \right)_{j=1}^{n_i}, \left( V_{j,i} \right)_{j=1}^{n_i} \right\}$ where $n_j $ is in $ \NN$ and denoting by $T_i$ the polygon  $A_{1,i}\ldots A_{n_i,i}$, we say that the sequence $\left( PN_i \right)_{i \in I}$ is strictly increasing if for all $i $ in $ I$, $T_i \subset \mathrm{int}(T_{i+1})$.

Let $\left( PN_i \right)_{i \in I}$ be a strictly increasing sequence of interpolable polygon-normal pairs. A sequence  $(\epsilon_i)_{i \in I}$ in $(0,1)$ is said to be {\em admissible} if $0 \in \mathrm{int}(T_i)$ for each $i $ in $ I$ and 
$$\gamma T_i\subset\inte T_{i+1}$$
for all $\gamma \in [1-\epsilon_i,1+\epsilon_i].$
We have the following corollary of Theorem \ref{th:smoothinterp}.

\begin{corollary}[Smooth convex interpolation of polygon sequences]
				Let $I = \ZZ$ or $I = \NN$. 
				Let $\left( PN_i \right)_{i \in I}$ be a strictly increasing sequence of interpolable polygon-normal pairs and $(\epsilon_i)_{i \in I}$ be admissible. Set $\displaystyle\mathcal{T}:=\mathrm{int}\left(\cup_{i\in I}T_i  \right).$
		
				Then for any $k $ in $ \NN$, $k \geq 2$ there exists a $C^k$ convex function $f \colon \mathcal{T} \mapsto \RR$, and an increasing sequence $(\lambda_i)_{i \in I}$, with $\inf_{i \in I} \lambda_i > -\infty$, such that for each $i $ in $ I$
				\begin{itemize}
								\item[(i)] $T_i \subset \left\{ x,\,f(x) \leq \lambda_i \right\}$.
								\item[(ii)] $\dist(T_i, \left\{ x,\,f(x) \leq \lambda_i \right\}) \leq \epsilon_i$.
								\item[(iii)] For each $i $ in $ I$, $j$ in $\{1,\ldots,n_i\}$, we have  $f(A_{i,j}) = \lambda_i$ and $\nabla f(x)$ is colinear to $V_{i,j}$.
								\item[(iv)] $\nabla^2 f$ is positive definite outside $\argmin f$. When there is a unique minimizer then $\nabla^2f$ is positive definite  throughout $\T$  {\rm (}this is the case when $I = \NN$ or when $I=\ZZ$ and $\cap_{i \in I} T_i$ is a singleton{\rm )}.
				\end{itemize}
				\label{cor:polygonDecrease}
\end{corollary}

We add two remarks which will be useful for directional convergence issues and the construction of Legendre functions:
\begin{itemize}
				\item[(a)] If two consecutive elements of the sequence of interpolable polygon-normal pairs are homothetic with center $0$ in the interior of both polytopes, then the restriction of the resulting convex function to this convex ring can be constructed such that all the sublevel sets within this ring are homothetic with the same center. 
				\item[(b)] If further conditions are imposed on the elements of a strictly increasing interpolable polygon-normal pair, then the resulting function can be constructed to be Legendre and globally Lipschitz continuous (that is, its Legendre conjugate has bounded support). This is a consequence of Proposition \ref{th:globallyLipshitz} and will be detailled in the next section. 
\end{itemize}

\subsection{More on Legendre functions and a pathological function with polyhedral domain}

Using intensively polygonal interpolation, we build below a finite continuous Legendre function $h$ on an $\ell^\infty$ square with oscillating ``mirror lines": $t\to \nabla h^*(\nabla h(x_0)+tc)$.

We start with the following preparation proposition related to the Legendre interpolation of Proposition~\ref{th:globallyLipshitz}.
\begin{lemma}
				Let $\left( PN_i \right)_{i \in \NN^*}$ be a strictly increasing sequence of interpolable polygon-normal pairs. 
				Setting for $i$ in  $\NN^*$, $PN_i = \left\{ \left( A_{j,i} \right)_{j=1}^{n_i}, \left( V_{j,i} \right)_{j=1}^{n_i} \right\}$ where $n_j$ is in $\NN^*$ and denoting by $T_i$ the polygon $A_{1,i}\ldots A_{n_i,i}$, we assume that $$T_i=3i P,\,\forall i \in \NN^*,$$ where $P$ is a fixed polygon which contains the unit Euclidean disk.\\ Then for any $l$ in $\NN$, $l \geq 2$, there exists a strictly increasing sequence of sets $\left( S_i \right)_{i \in \NN, \,i\geq 2}$, such that for $j \geq 1$, 
				\begin{itemize}
								\item $S_{3j}$ interpolates the normals of $PN_j$ in the sense of Lemma \ref{lem:polyGon} with $\mathrm{dist}(S_{3j}, T_j) \leq 1 / (4(3j+2)^l)$
								\item $S_{3j - 1} = \frac{3j - 1}{3_j} S_{3j}$
								\item $S_{3j + 1} = \frac{3j + 1}{3_j} S_{3j}$
				\end{itemize}
				This sequence has the following properties\begin{itemize} \item there exists $c>0$, such that for all $j $ in $\NN$, $j \geq 3$ and for all unit vector $x$, 
				\begin{align}
				    c\geq\sigma_{S_{j+1}}(x) - \sigma_{S_j}(x) \geq 1 - \frac{1}{(j+1)^l}.
				    \label{eq:interpolGauge1}
				\end{align}
				 \item  for all unit vector $x$,
				\begin{align}
								\left| \frac{\sigma_{S_{j+1}}(x)  - \sigma_{S_j}(x)}{ \sigma_{S_{j}}(x)  - \sigma_{S_{j-1}}(x)} -1\right| \leq \frac{1}{j^l},\:\forall j\geq 3.
								\label{eq:interpolGauge2}
				\end{align}
\end{itemize}
				\label{lem:interpolGauge}
\end{lemma}
\begin{proof}
				Set for all $j$  in $\NN^*$, $\delta_j = \frac{1}{4(3j+2)^l}$ and let $S_{3j}$ be the $\delta_j$ interpolant of $T_j=3jP$ as given by Lemma \ref{lem:polyGon} so that $\dist(S_{3j}, 3jP) \leq \delta_j$.  Since $P$ contains the unit  ball, \begin{equation}\label{inc} 3j P \subset S_{3j} \subset (3j + \delta_j) P.\end{equation} Now set
				\begin{align*}
								S_{3j - 1} &= \frac{3j - 1}{3_j} S_{3j}\\
								S_{3j + 1} &= \frac{3j + 1}{3_j} S_{3j}.
				\end{align*}
				For any $j$ in $\NN^*$, it is clear that $S_{3j-1} \subset \mathrm{int}(S_{3j})$ and $S_{3j} \subset \mathrm{int}(S_{3j+1})$. Furthermore, by \eqref{inc},  we have 
				\begin{align*}
								S_{3j+1} \subset \frac{3j+1}{3j} ( 3j + \delta_j) P \subset ((3j+1) + 2\delta_j) P \subset \mathrm{int}((3j+2) P) \subset \mathrm{int} (S_{3j+2})
				\end{align*}
				so that we indeed have a strictly increasing sequence of sets. We obtain from the construction, for any $j$ in  $\NN^*$, and any  unit vector $x$,
				\begin{align}
								\sigma_{S_{3j}}(x) - \sigma_{S_{3j-1}}(x) &=\sigma_{S_{3j+1}}(x) - \sigma_{S_{3j}}(x)\notag \\
								&= \frac{1}{3j} \sigma_{S_{3j}}(x) \in \left[ \sigma_P(x), \left( 1 + \frac{\delta_j}{3j} \right) \sigma_P(x) \right] \subset \left[ 1, \left(1+\frac{\delta_j}     {3j} \right) \sigma_P(x) \right]\notag \\
								\sigma_{S_{3j+2}}(x) - \sigma_{S_{3j+1}}(x) &\leq \sigma_P(x) (3j +2) \left( 1 + \frac{\delta_{j+1}}{3j+3} \right) - \sigma_P(x)(3j+1)\notag\\
								\sigma_{S_{3j+2}}(x) - \sigma_{S_{3j+1}}(x) &\geq \sigma_P(x) (3j +2) - \sigma_P(x)(3j+1)\left( 1 + \frac{\delta_j}{3j} \right) \notag\\
								& = \sigma_P(x) \left( 1 - \delta_j \frac{3j+1}{3j} \right)\notag\\ 
								&\geq 1 -  \delta_j \frac{4}{3}  \geq 1 - \frac{1}{(3j+2)^l}\label{intermed}
				\end{align}
				This proves \eqref{eq:interpolGauge1}. 
				We deduce that for all $j$  in $\NN^*,$
				\begin{align*}
								\frac{\sigma_{S_{3j+1}}(x)  - \sigma_{S_{3j}}(x)}{ \sigma_{S_{3j}}(x)  - \sigma_{S_{3j-1}}(x)}& = 1, \mbox{ for all nonzero vector $x$,} \\
								\max_{\|x\| = 1} \frac{\sigma_{S_{3j+2}}(x)  - \sigma_{S_{3j+1}}(x)}{ \sigma_{S_{3j+1}}(x)  - \sigma_{S_{3j}}(x)} &\leq (3j +2) \left( 1 + \frac{\delta_{j+1}}{3j+3} \right) - (3j+1)\\
								&= 1 + \frac{3j+2}{3j+3} \delta_{j+1} \leq 1 + \frac{1}{(3j +1)^l},\\
								\min_{\|x\| = 1} \frac{\sigma_{S_{3j+2}}(x)  - \sigma_{S_{3j+1}}(x)}{ \sigma_{S_{3j+1}}(x)  - \sigma_{S_{3j}}(x)} &\underset{\eqref{intermed}}{\geq} \frac{1 - \delta_j \frac{3j+1}{3j} }{\left( 1 + \frac{\delta_j}{3j} \right) }\\
								&\geq \left( 1 - \delta_j \frac{3j+1}{3j}  \right) \left( 1 - \frac{\delta_j}{3j} \right) \\ 
								& \geq 1 - \delta_j\left( \frac{3j+1}{3j} + \frac{1}{3j} \right) \geq 1 - \delta_j \frac{5}{3} \geq 1 - \frac{1}{(3j+1)^l}.
				\end{align*}
				Furthermore, using the fact that $t \mapsto \frac{1+t}{1-t}$ is increasing on $(-\infty,1)$ and the fact that $\delta_{j+1} \leq \delta_j$,
				\begin{align}
								\max_{\|x\| = 1} \frac{\sigma_{S_{3j+3}}(x)  - \sigma_{S_{3j+2}}(x)}{ \sigma_{S_{3j+2}}(x)  - \sigma_{S_{3j+1}}(x)} &\underset{\eqref{intermed}}{\leq} \frac{1 + \frac{\delta_{j+1}}{3j + 3}}{1 - (3j+1) \frac{\delta_j}{3j}} \notag\\
								&\leq\frac{1 + (3j+1)\frac{\delta_{j}}{3j}}{1 - (3j+1) \frac{\delta_j}{3j}} \notag\\
								&\leq\frac{1 + \delta_j\frac{4}{3}}{1 - \delta_j \frac{4}{3}}\label{borntobewild}.
				\end{align}
				Setting $s(t)= (1+t)/(1-t)$, we have, for all $t \leq 1/2$
				\begin{align*}
s'(t) &= \frac{2}{(1-t)^2}, \, s(0)=1\\
					s''(t) &= \frac{4}{(1-t)^3} \leq 24,\, s'(0)=2.
				\end{align*}	Thus   $s(t) \leq 1 + 2 t + 12 t^2$ on $(-\infty,1/2]$.	Since $\frac{4}{3} \delta_j \leq \frac{4}{75} \leq \frac{1}{2}$, we deduce from the previous remark and \eqref{borntobewild} above:
				\begin{align*}
									\max_{\|x\| = 1} \frac{\sigma_{S_{3j+3}}(x)  - \sigma_{S_{3j+2}}(x)}{ \sigma_{S_{3j+2}}(x)  - \sigma_{S_{3j+1}}(x)} &\leq 1 + \frac{8}{3} \delta_j + \frac{64}{3} \delta_j^2 = 1 + \delta_j\left( \frac{8}{3} + \frac{64}{3} \delta_j \right)\\
								&\leq 1 + \delta_j\left(3 + \frac{64}{3 \times 25}\right) \leq 1 + 4 \delta_j = 1 + \frac{1}{(3j+2)^l}.
				\end{align*}
				Finally using \eqref{intermed} again,
				\begin{align*}
								\min_{\|x\| = 1} \frac{\sigma_{S_{3j+3}}(x)  - \sigma_{S_{3j+2}}(x)}{ \sigma_{S_{3j+2}}(x)  - \sigma_{S_{3j+1}}(x)} &\geq \frac{1 }{(3j +2) \left( 1 + \frac{\delta_{j+1}}{3j+3} \right) - (3j+1) } \\
								&= \frac{1 }{1 + \delta_{j+1}\frac{3j+2}{3j+3}  } \\
								&\geq 1 - \delta_{j+1}\frac{3j+2}{3j+3} \geq 1 - \delta_{j+1} \geq 1 - \frac{1}{(3j+2)^l}. \end{align*}
				This proves the desired result.
\end{proof}

Combining Lemma \ref{lem:interpolGauge} and Proposition \ref{th:globallyLipshitz}, we obtain the following result.
\begin{theorem}
Let $\left( PN_i \right)_{i \in \NN^*}$ be a  strict\-ly increasing sequence of interpolable po\-ly\-gon-nor\-mal pairs.  
				Set  for $i$ in  $\NN^*$, $PN_i = \left\{ \left( A_{j,i} \right)_{j=1}^{n_i}, \left( V_{j,i} \right)_{j=1}^{n_i} \right\}$ where $n_i$  is in $\NN^*$,  denote by $T_i$ the polygon  $A_{1,i}\ldots A_{n_i,i}$, and
				assume that for all $i $ in $ \NN^*$, $T_i=3i P$ where $P$ is a fixed polygon which contains the unit disk in its interior.\\ 
				Then for any $k $ in $ \NN$, $k \geq 2$ and all $l \geq 3$, there exists a $C^k$ globally Lipschitz continuous Legendre function, $h \colon \RR^2 \mapsto \RR$, and an increasing sequence $(\lambda_i)_{i \in  \NN}$, such that for each $i $ in $ \NN$: \begin{itemize}
								\item $T_i \subset \left\{ x,\,h(x) \leq \lambda_i \right\}$,
								\item $\dist(T_i, \left\{ x,\,h(x) \leq \lambda_i \right\}) \leq \frac{1}{4(3i+2)^l}$,
								\item For any $x$ with $h(x) = \lambda_i$ and $\nabla h(x)$ is colinear to $V_i$ for each vertex $x$ of $T_i$,
								\item $h$ has positive definite Hessian and is globally Lipschitz continuous, 
								\item $h^*$ has compact domain and is $C^k$ on the interior of its domain.
				\end{itemize}
				\label{th:Legendre}
\end{theorem}

\begin{corollary}[Continuity on the domain]
                The function $h^*$ constructed in Theorem \ref{th:Legendre} has compact polygonal domain and is continuous on this domain.
                \label{cor:legendreContinuity}
\end{corollary}
\begin{proof}
    Since $P$ is a polygon and contains the unit Euclidean disk, the gauge function of $3P$ is polyhedral with full domain $\RR^2$, call it $\omega$.  Denote by $P^{\circ}$ the polar of $P$. This is a polytope and since $\omega$ is the gauge of $P$, we actually have $\omega = \sigma_{P^{\circ}}$, the support function of the polar of $P$ \cite[Theorem 1.7.6]{schneider1993convex}. Hence the the convex conjugate of $\omega$ is the indicator of the polytope $P^\circ$ \cite[Theorem 13.2]{rockafellar1970convex}.
    
    It can be easily seen from the proof of Proposition \ref{th:globallyLipshitz} that $\lambda_i = \alpha i + r_i$ with $r(i) = O(1)$ as $i \to \infty$. Without loss of generality, we may suppose that $\alpha = 1$ (this is a simple rescaling) so that there is a positive constant $c$ such that $|\lambda_i - i| \leq c$ for all $i$. 
    
    Let $h$ be given as in Theorem \ref{th:Legendre}, fix $i \geq 1$ and $x \in \RR^2$ such that  $\lambda_{i-1} \leq h(x) \leq \lambda_i$. We have in $\R^2$
    \begin{align*}
       \{y:\, \omega(y) \leq i-1 \} \subset \{y:\, h(y) \leq  \lambda_{i-1} \} \subset  \{y:\, h(y) \leq \lambda_{i} \} \subset \{y:\, \omega(y) \leq i+1 \} 
    \end{align*}
    and hence
    \begin{align*}
        i - 1 \leq \omega(x) \leq i+1
    \end{align*}
    and we deduce that
    \begin{align*}
        \omega(x) - 2 - c \leq i - 1 -c \leq  \lambda_{i-1} \leq h(x) \leq \lambda_i \leq i+c \leq \omega(x) + c + 1.
    \end{align*}
    Since $i$ was arbitrary, this shows that there exists a constant $C>0$ such that $|h(x) - \omega(x)|\leq C$ for all $x \in \RR^2$. Recall that $z \mapsto \sup_{y \in \RR^2} \left\langle y,z\right\rangle - \omega(y)$ is the indicator function of $P^\circ$, hence,
    \begin{align*}
        z \in P^{\circ} & \Rightarrow \sup_{y \in \RR^2} \left\langle y,z\right\rangle - \omega(y) = 0 \Rightarrow \sup_{y \in \RR^2} \left\langle y,z\right\rangle - h(y) \leq C < +\infty \\
        z \not\in P^{\circ} & \Rightarrow \sup_{y \in \RR^2} \left\langle y,z\right\rangle - \omega(y) = +\infty \Rightarrow \sup_{y \in \RR^2} \left\langle y,z\right\rangle - h(y) = +\infty
    \end{align*}
    which shows that the domain of $h^*$ is actually $P^{\circ}$ which is a polytope. Now, $h^*$ is convex and lower semicontinuous on $P^{\circ}$, invoking the results of \cite{gale68convex}, it is also upper semicontinuous on $B^*$ and finally it is continuous on $B^*$.
\end{proof}

\begin{corollary}[A pathological Legendre function]
    For any $\theta \in \left( \frac{-\pi}{4},\frac{\pi}{4} \right)$ there exists a Legendre function $h\colon \RR^2 \mapsto \RR$ whose domain is a closed square, continuous on this domain and $C^k$ on its interior, such that for all $i \in \NN^*$, $\nabla h^*(i, 0)$ is proportional to $(\cos(\theta), (-1)^i \sin(\theta))$.
    \label{cor:legendreOscillating}
\end{corollary}
\begin{proof}
				For $x=(u,v)$, set $\|x\|_1=|u|+|v|$, and  let $P = \left\{ x \in \RR^2,\, \|x\|_1 \leq 2 \right\}$. Let us construct  a strictly increasing sequence of interpolable polygon-normal pairs $\left( PN_i \right)_{i \in \NN^*}$ as follows, we fix $\theta \in \left( \frac{-\pi}{4},\frac{\pi}{4} \right)$ and set for all $i\in \NN^*$ :
				\begin{itemize}
								\item $T_i = 3i P$, the polygon associated to the $i$-th term  $PN_i$ of the sequence,
								\item except at the rightmost corner, consider the normals given by the canonical basis vectors and their opposite,
								\item at the rightmost corner, $(6i,0)$, one chooses the normal given by the vector $$(\cos(\theta), (-1)^i \sin(\theta)).$$
				\end{itemize}
				We now invoke Theorem \ref{th:Legendre} to obtain a Lipschitz continuous Legendre function, denoted $h^*$, with full domain having all the $T_i$ as sublevel sets and satisfying the hypotheses of the corollary. Rescaling by a factor $6$ and setting $h = h^{**}$ gives the result.
\end{proof}

\section{Counterexamples in continuous optimization}

We are now in position to apply our interpolation results to build counterexamples to  classical problems in convex optimization. We worked on situations ranging from structural questions to qualitative behavior of algorithms and ODEs. Through  9 counterexamples we tried to cover a large spectrum but there are many more possibilities that are left for future research. Some example are constructed from decreasing sequences of convex sets, they can be interpolated using Theorem \ref{th:smoothinterp} with $I = \ZZ$, indexing the sequence with negative indices and adding artificially additional sets for positive indices. Nonetheless we sometimes depart from the notations of the first sections and index these sequences by $\NN$ even though they are decreasing for simplification purposes. 

\subsection{Kurdyka-\L ojasiewicz inequality may not hold}
The following result is proved in \cite{bolte2010characterization}, it was crucial to construct a $C^2$ convex function which does not satisfy Kurdyka-\L ojasiewicz (KL) inequality.  
\begin{lemma}{\cite[Lemma 35]{bolte2010characterization}}
				There exists a decreasing sequence of compact convex sets $\left( T_i \right)_{i\in \NN}$ such that for any $i $ in $ \NN$, $T_- = T_{i+1}$ and $T_+ = T_i$ satisfy Assumption \ref{ass:curvature} and
				\begin{align*}
								\sum_{i=0}^{+\infty} \dist(T_i, T_{i+1}) = +\infty
				\end{align*}
				\label{lem:tamsLoja}
\end{lemma}
As a corollary, we improve the   counterexample in \cite{bolte2010characterization} and provide a $C^k$ convex coun\-ter\-examples for any $k\geq 2 $ in $ \NN$. 
\begin{corollary}[Smooth convex functions are not KL in general]
                There exists a $C^k$ convex function $f \colon \RR^2 \mapsto \RR$ which does not satisfy KL inequality. More precisely, for any $r > \inf f$ and $\varphi \colon [\inf f, r] \mapsto \RR$ continuous and differentiable on $(\inf f, r)$ with $\varphi' > 0$ and $\varphi(\inf f) = 0$, we have
                \begin{align*}
                    \inf \{\| \nabla (\varphi \circ f)(x) \|: \: x\in \R^2, \,\inf f < f(x)  < r \} = 0.
                \end{align*}
\end{corollary}
\begin{proof}
				Using \cite[Theorem 1.8.13]{schneider1993convex}, each $T_i$ can be approximated up to arbitrary precision by a polygon. Hence we may assume that all $T_i$ are polygonal while preserving the property of Lemma \ref{lem:tamsLoja} as well as Assumption \ref{ass:curvature}. Furthermore, using Lemma \ref{lem:polyGon} and Remark \ref{rem:noNormal} each $T_i$ can in turn be approximated with arbitrary precision by a convex set with $C^k$ boundary and positive curvature. Hence we may also assume that all $T_i$ satisfy both the result of Lemma \ref{lem:tamsLoja} and have $C^k$ boundary with nonvanishing curvature. Reversing the order of the sets and adding additional sets artificially, we are in the conditions of application of Theorem \ref{th:smoothinterp} with $I = \ZZ$ and the resulting $f$ follows from the same argument as in \cite[Theorem 36]{bolte2010characterization}.
\end{proof}

\subsection{Block coordinate descent  may not converge}
\begin{figure}[H]
				\centering
				\includegraphics[width=.45\textwidth]{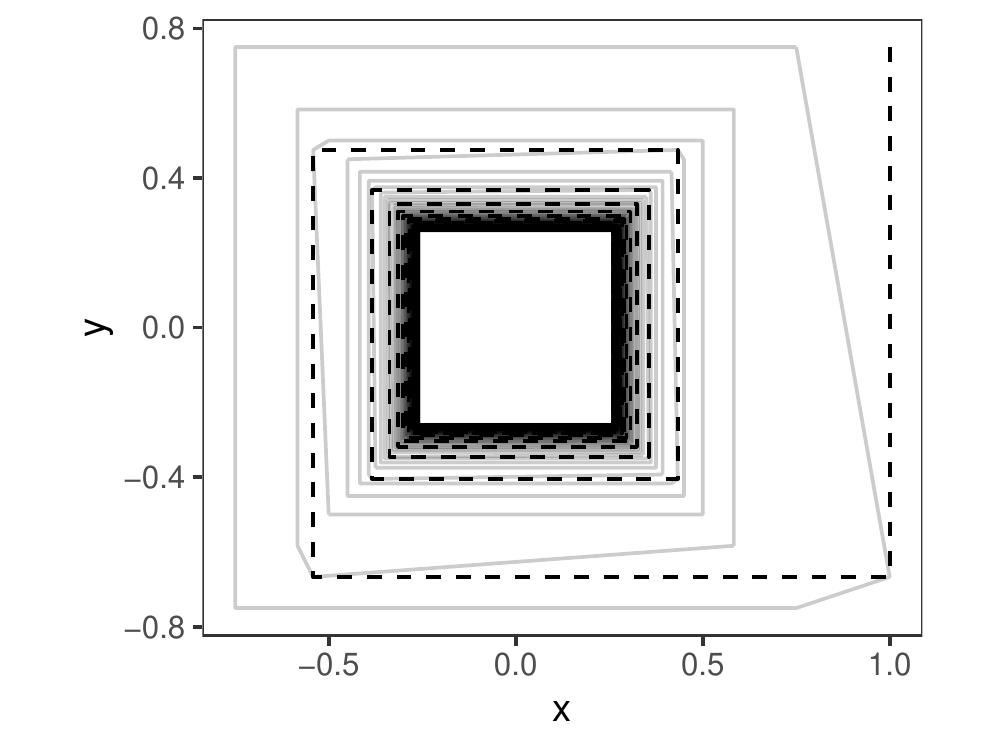}\quad \includegraphics[width=.4\textwidth]{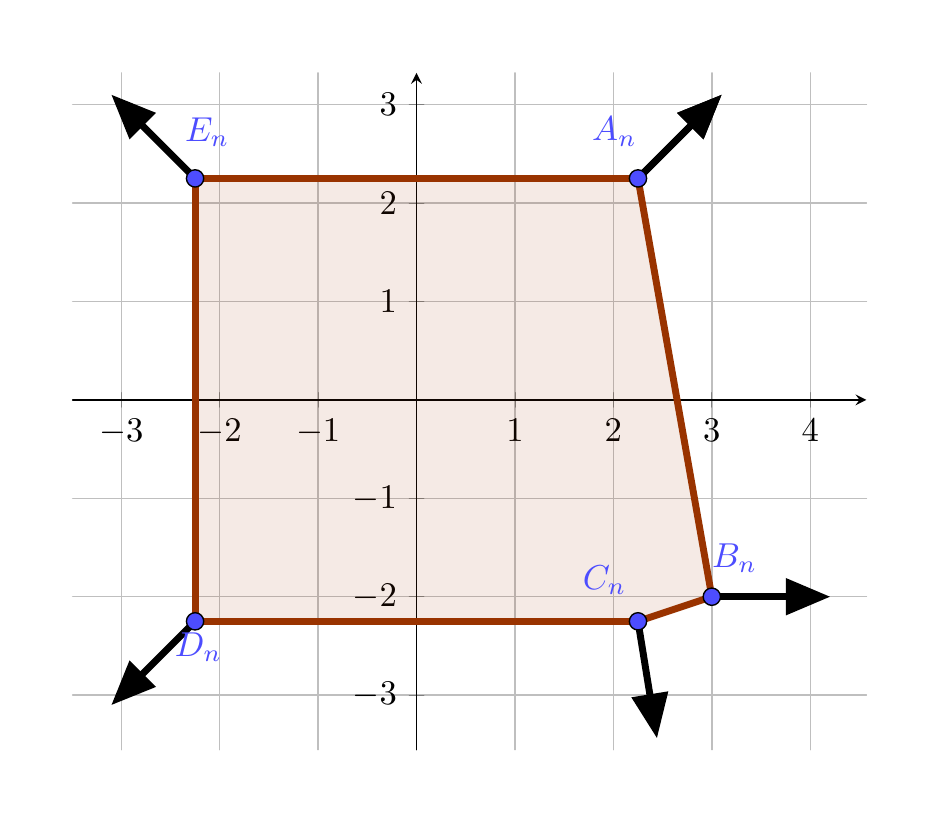}
				\caption{Illustration of the alternating minimization (resp. exact line search) example: on the left, the sublevel sets in gray and the corresponding alternating minimization (resp. exact line search) sequence in dashed lines. On the right the  interpolating polygons together with their normal vectors as in Lemma \ref{lem:polyGon}.}
				\label{fig:illustrAltMin}
\end{figure}
The following polygonal construction is illustrated in Figure \ref{fig:illustrAltMin}. 
For any $n\geq2$ in $\NN$, we set
\begin{align*}
				A_n &= \left(\frac{1}{4} + \frac{1}{n},   \frac{1}{4} + \frac{1}{n}\right)\\
				B_n &= \left( \frac{1}{4} + \frac{1}{2(
			n-1)} + \frac{1}{2n}, \frac{1}{4} + \frac{1}{2n} + \frac{1}{2(n+1)} \right)\\
				C_n &= \left(\frac{1}{4} + \frac{1}{n},  - \frac{1}{4} - \frac{1}{n}\right)\\
				D_n &= \left(-\frac{1}{4} - \frac{1}{n},  - \frac{1}{4} - \frac{1}{n}\right)\\
				E_n &= \left(-\frac{1}{4} - \frac{1}{n},  + \frac{1}{4} + \frac{1}{n}\right).
\end{align*}
This defines a convex polygon. We may choose the normals at $A_n, C_n, D_n, E_n$ to be bisectors of the corresponding corners and the normal at $B_n$ to be horizontal (see Figure \ref{fig:illustrAltMin}). Rotating by an angle of $-\frac{n\pi}{2}$ and repeating the process indefinitely, we obtain the sequence of polygons depicted in Figure \ref{fig:illustrAltMin}. It can be checked that the polygons form a strictly decreasing sequence of sets, as for $n>1$, the polygon $A_nB_nC_nD_nE_n$ is contained in the interior of the square $A_{n-1}C_{n-1}D_{n-1}E_{n-1}$. This fulfills  the requirement of Corollary~\ref{cor:polygonDecrease}. 
\begin{corollary}
				There exists a $C^k$ convex function $f \colon \RR^2 \mapsto \RR$ and an initialization $x_0=(u_0,v_0)$ such that the recursion, for $i \geq 1$
				\begin{align*}
								u_{i+1} &\in \argmin_{u} f(u,v_i) \\
								v_{i+1} &\in \argmin_{v} f(u_{i+1},v)
				\end{align*}
				produces a non converging sequence $\displaystyle (x_i)_{i\in\NN}=\left((u_i,v_i)\right)_{i\in\NN}$.
				\label{cor:altMin}
\end{corollary}
\begin{proof}
 We apply  Corollary~\ref{cor:polygonDecrease} to the proposed decreasing sequence and by choosing $(u_0,v_0) = B_2$ for example. This requires to shift indices (start with $i = 2$) and use Theorem \ref{th:smoothinterp} with $I = \ZZ$. Note that the optimality condition for partial minimization and the fact that level sets have  nonvanishing curvature ensure that the partial minima are unique. 
\end{proof}

In the nonsmooth convex case cyclic minimization is known to fail to provide the infimum value, see e.g., \cite[p. 94]{auslender}.  Smoothness is sufficient for establishing value convergence (see e.g. \cite{beck2013convergence,wright2015coordinate} and references therein), whether it is enough or not for obtaining convergence was an open question. Our counterexample closes this question and shows that cyclic minimization does not  yield converging sequences even for $C^k$ convex functions. This result also closes the question for the more general nonconvex case for which we are not aware of a nontrivial  counterexample for convergence of alternating minimization. Let us mention however Powell's example \cite{powell1973search} which shows that cyclic minimization with three blocks does not converge for smooth functions. 

It would also be interesting to understand how our result may impact dual methods and counterexamples in that field, as for instance the recent three blocks counterexample in \cite{chen}.

\subsection{Gradient descent with exact line search may not converge}

Gradient descent with exact line search is governed by the recursion: $$x^+\in \argmin \left\{f(y):y=x-t\nabla f(x),\,t\in\RR\right\},$$
where $x$ is a point in the plane.

Observe that the step coincides with partial minimization when the gradient $\nabla f(x)$ is colinear to one of the axis of the canonical basis. From the previous section, we thus deduce the following.
\begin{corollary}[Failure of gradient descent with exact line search]
				There exists a $C^k$ convex function $f \colon \RR^2 \mapsto \RR$ and an initialization $z_0$ in the plane  such that the recursion, for $i \geq 1$
				\begin{align*}
	x_{i+1} \in \argmin \left\{f(y): y=x_i-t\nabla f(x_i),\, t\in\R\right\}
		\end{align*}
				produces a well defined  non converging sequence $\left(x_i \right)_{i \in \NN}$.
				\label{cor:exactLineSearch}
\end{corollary}

 Convergence failure for  gradient descent with exact line search is new up to our knowledge. Let us mention that despite non convergence, the constructed sequence satisfy sublinear convergence rates in function values \cite{beck2013convergence}.

\subsection{Tikhonov regularization path may have infinite length}

\begin{figure}[]
				\centering
				\includegraphics[width=.38\textwidth]{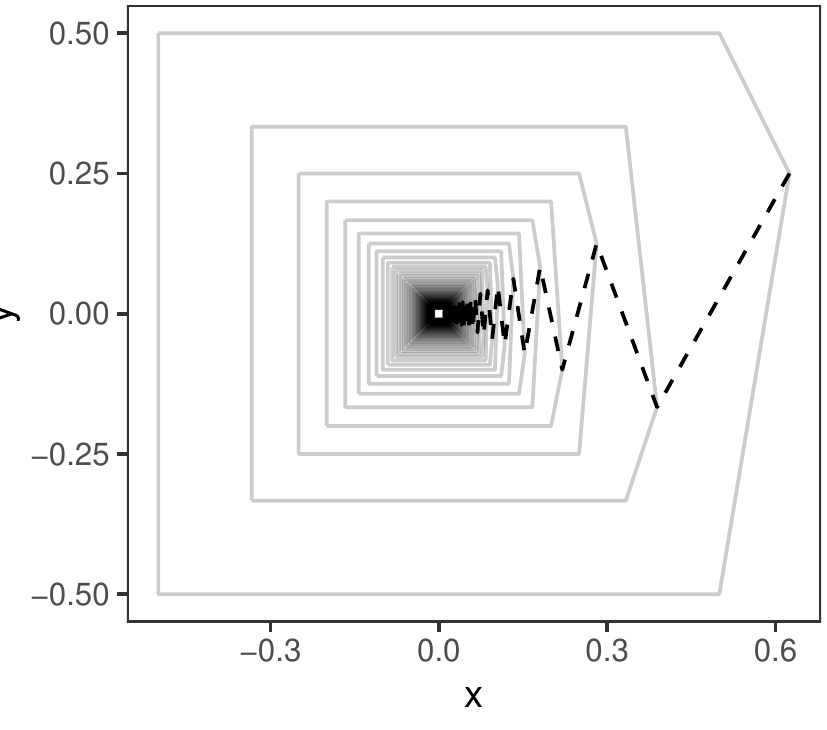}\quad \includegraphics[width=.55\textwidth]{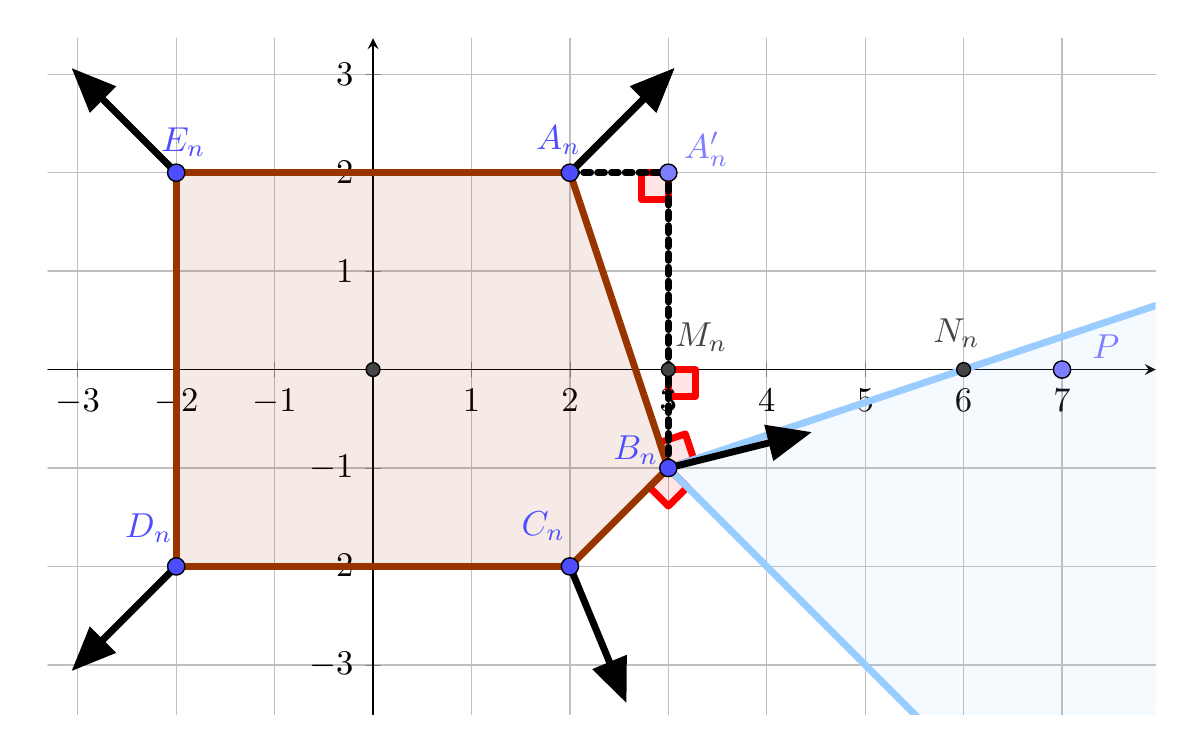}
				\caption{Illustration of the Tikhonov regularization example, on the left in gray, polygons used to build the sublevel sets of the constructed $f$  and the corresponding solutions to \eqref{eq:tikhonovRegularization} for some  values of $r$ (solutions are joined by dotted lines). On the right the normal to be chosen to apply Lemma \ref{lem:polyGon} (for $n = 1$, see main text for details). The point $P$ represents $\bar x$,  it sits on the $x$-axis and is constantly contained in the normal cone at $B_n$ for any $n \geq 1$.}
				\label{fig:regTorralba}
\end{figure}
Following \cite{torralba96}, we consider for any $r > 0$
\begin{align}
				x(r) = \argmin\left\{ f(x) + r \|x - \bar{x}\|_2^2:x\in \RR^2\right\}
				\label{eq:tikhonovRegularization}
\end{align}
where $f$ is  $C^k$ convex and where $\bar x$ is any anchor point. We would like to show that the curve $r \mapsto x(r)$ may have infinite length. Torralba provided a counterexample in his PhD Thesis for {\em continuous} convex functions, see  \cite{torralba96}. This work extends his result to  smooth $C^k$ convex functions in $\RR^2$. 

 For any $n $ in $ \NN^*$, we set
\begin{align*}
				A_n &= \left(\frac{2}{n}, \frac{2}{n}\right)\\
				B_n &= \left(\frac{2}{n} + \frac{1}{n^2}, -\frac{1}{n}  \right)\\
				C_n &= \left(\frac{2}{n},  - \frac{2}{n}\right)\\
				D_n &= \left(- \frac{2}{n}, - \frac{2}{n}\right)\\
				E_n &= \left(- \frac{2}{n}, \frac{2}{n}\right).
\end{align*}
This is depicted in Figure \ref{fig:regTorralba}. For all $n \geq 1$, denote by $M_n$ the point on the $x$ axis above $B_n$ and $N_n$, the intersection of the normal cone at $B_n$ and the $x$ axis. We have
\begin{align*}
    \frac{M_nN_n}{M_nB_n} = n \times M_nN_n= \frac{A'_nB_n}{A_nA'_n} = \frac{3/n}{1/n^2} = 3 n,
\end{align*}
so that	for all $n \geq 1$, $M_nN_n = 3$ and $N_n = (3 + 2/n + 1 / n^2, 0)$. Choosing $P = (7,0)$, since for $n \geq 1$, $3 + 2/n + 1 / n^2 \leq 6 < 7$ , this shows that $P$ constantly belongs to the interior of the normal cone at $B_n$ for all $n \geq 1$. The sequence of level sets is constructed as in Figure~\ref{fig:regTorralba} by considering alternating symmetries with respect to the $x$-axis of the sequence of polygons above. It can be checked that the polygons form a strictly decreasing sequence of sets, as for $n>1$, the polygon $A_nB_nC_nD_nE_n$ is contained in the interior of the square $A_{n-1}C_{n-1}D_{n-1}E_{n-1}$.  We choose the normal at $A_n, C_n, D_n, E_n$ to belong to the bisector at the corner and the normal at $B_n$ to be proportional to the vector $B_n P$. Applying Corollary \ref{cor:polygonDecrease}, we construct $f$ and choose $\bar{x} = P$ in \eqref{eq:tikhonovRegularization} to obtain the following:

\begin{corollary}[A bounded infinite length Tikhonov path]
				There exists a $C^k$ strictly convex function $f \colon \RR^2 \mapsto \RR$ and $\bar{x} \in \RR^2$ such that the curve $x((0,1))$ given by \eqref{eq:tikhonovRegularization} has infinite length.
				\label{cor:homotopy}
\end{corollary}
\begin{proof}
        We apply Corollary \ref{cor:polygonDecrease} with $I = \ZZ$ and revert the indices set to match the sequence that we have described. For any $n \geq 1$ there exists a value of $\lambda_n$ such that $f(B_n)= \lambda_n$ and $\nabla f(B_n)$ is colinear to the vector $B_n P$. Set 
        \begin{align*}
            r = \frac{\|\nabla f(B_n)\|}{2 B_nP}
        \end{align*}
        we have $\nabla f(B_n) + 2r(B_n - P) = 0$ which is the optimality condition in \eqref{eq:tikhonovRegularization} with $\bar{x} = P$. Hence we have shown that there exists a value of $r$ such that $B_n$ is the solution to \eqref{eq:tikhonovRegularization}. Since $n$ was arbitrary this is true for all $n$ and the curve $r \mapsto x(r)$ has to go through a sequence of points whose second coordinate is of the form $(-1)^n/n$ for all $n \geq 1$. Since this sequence is not absolutely summable, the curve has infinite length.
\end{proof}

This result is in contrast with the definable case for which we have finite length by the monotonicity lemma, since the whole trajectory is definable and bounded. 

\subsection{Secants of gradient curves at infinity may not converge}
\paragraph{Thom's gradient conjecture and Kurdyka-Mostowski-Parusinski's theorem} A theorem of \L ojasiewicz \cite{lojasiewicz1984trajectoires} asserts that bounded solutions  to the gradient system
\begin{align*}
    \dot{x}(t) = - \nabla f(x(t))
\end{align*}
converge when $f$ is a real analytic potential. Thom conjectured in \cite{thom1989problemes} that this convergence should occur in a stronger form:  trajectories converging to a given $\bar{x}$ should admit a tangent at infinity, that is
\begin{align}\label{secants}
    \frac{x(t) - \bar{x}}{\|x(t) - \bar{x}\|}
\end{align}
should have a limit as $t \to \infty$. Lines passing through $\bar x$ and  having \eqref{secants} as a slope are called {\em secants of $x$ at $\bar x$}. This conjecture was proved to be true in \cite{kurdyka2000proof}. In the convex world, it is well known that solutions to the gradient system converge for general potentials (this is a F\'ejer monotonicity argument due to Bruck); see also the original approaches by Manselli and Pucci \cite{Manselli} and Daniilidis et al. \cite{dan}. It is then natural to wonder  whether this convergence satisfies  higher order rigidity properties as in the analytic case. The answer turns out to be negative in general yielding a quite mysterious phase portrait.

\paragraph{Absence of tangential convergence for convex potentials}

\begin{figure}[H]
				\centering
				\includegraphics[width=.6\textwidth]{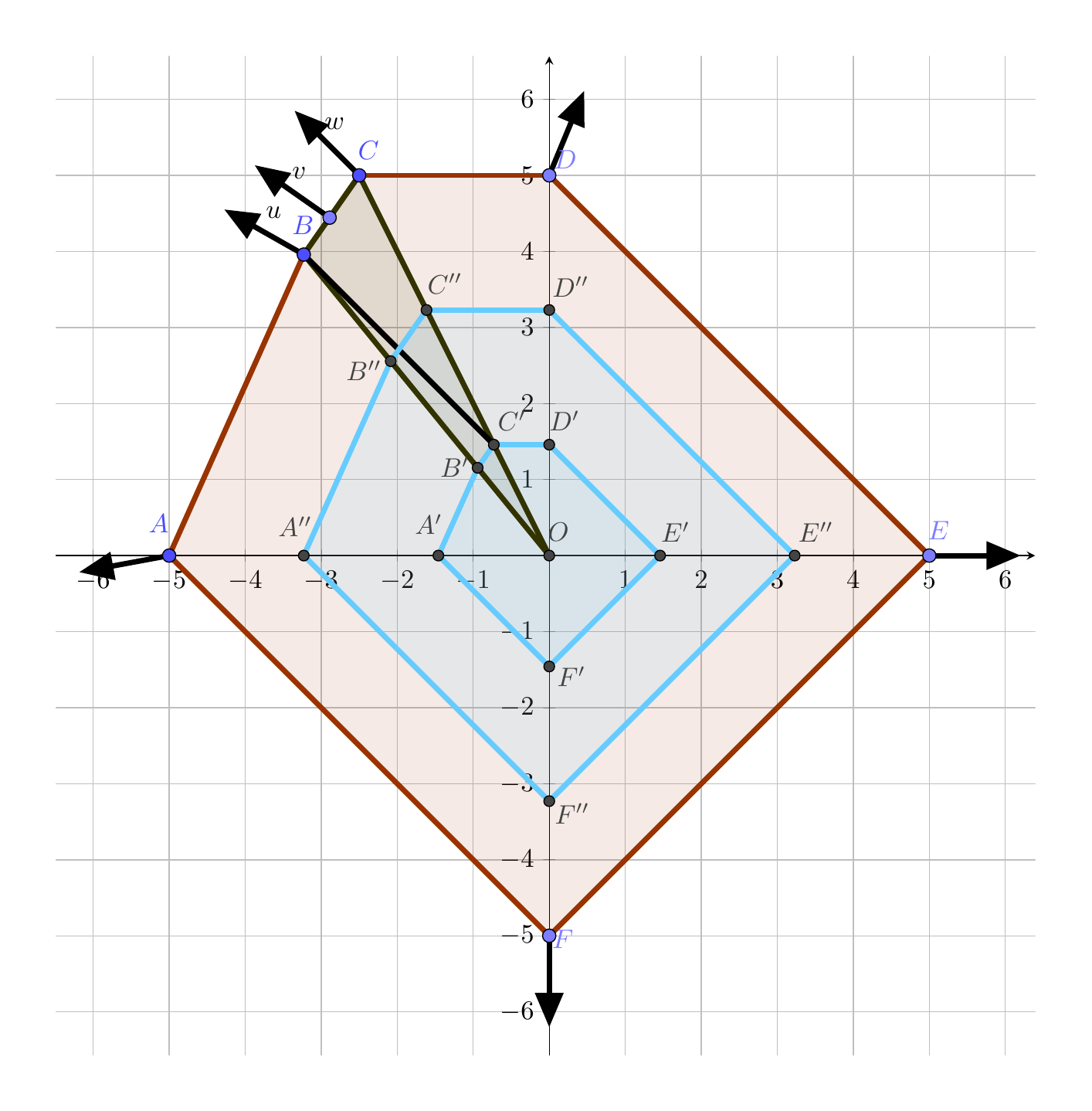}
				\caption{$A = (-5,0)$, $B=\left( - \frac{5}{3} - \frac{25}{16}, \frac{10}{3} + \frac{5}{4} \right)$, $C=\left( \frac{-5}{2},5 \right)$, $D=\left( 0,5 \right)$, $E = \left( 5,0 \right)$, $F = \left( 0, -5 \right)$. All normals are chosen to be bisectors except $w$ which is parallel to the line $(DE)$. The vector $v$ is orthogonal to the segment $[BC]$. The point $C'$ is obtained by considering the intersection between the line $(Bw)$ (starting from $B$ with direction $w$), and the segment $[OC]$. The points $A'$, $B'$, $D'$, $E'$, $F'$ are obtained by performing a scaling of $A,B,D,E,F$ of a factor $\frac{OC'}{OC}$. The polygon $A''B''C''D''E''F''$ is $ABCDEF$ scaled by a factor $\frac{OC' + OC}{2 OC}$.}
				\label{fig:thom}
\end{figure}
The construction given in this paragraph is more complex than the previous ones, we start with a technical lemma which will be the basic building block for our counterexample.
\begin{lemma}
				Let $S$ be a  convex set with $C^k$ boundary interpolating $ABCDEF$ in Figure \ref{fig:thom} and let $g$  be the gauge function associated to $S$. The function $g$ is differentiable outside the origin. Consider any initialization $x_0 $ in $ [BC]$ with corresponding trajectory solution to the  equation
				\begin{align*}
								\dot{x}(t) &= - \nabla g(x(t)), \, t\geq 0,\\
								x(0) &= x_0.
				\end{align*}
				Set $\bar{t} = \sup_{x(t) \in OBC} t$, we have $\bar{t} < + \infty$ and $x(\bar{t}) $ in $ [CC']$.
				\label{lem:techThom}
\end{lemma}
\begin{proof}   
                The fact that $g$ is differentiable comes from the fact that its subgradient is uniquely determined by the normal cone to $S$ which has dimension one because  of the smoothness of the boundary of $S$.  Since $S$ is interpolating the polygon, we have $g(B) = g(C) = 1$. Furthermore, we have for all $t$, $\frac{d}{dt} g(x(t)) = -\|\nabla g(x(t))\|^2= -1$, thence  $\bar{t} \leq 1 - g(C')$.
				By homogeneity, for any $x \neq 0$ and $s > 0$, $\nabla g (sx) = \nabla g(x)$. For any $x $ in $ [BC]$, by convexity
				\begin{align*}
				                0 \leq \left\langle C-x,\nabla g(C) -  \nabla g(x) \right\rangle = \left\langle C-B,\nabla g(C) -  \nabla g(x)  \right\rangle \frac{\|C-x\|}{\|C-B\|},
				\end{align*}
				and therefore 
				\begin{align}
								-\left\langle C - B, \nabla g(x) \right\rangle \geq  -\left\langle C - B, \nabla g(C) \right\rangle  
								\label{eq:thomGradient1}
				\end{align}
				By homogeneity of $g$, \eqref{eq:thomGradient1} is true for any $x$ in the triangle $OCB$ (different from $0$) and thus in the triangle $C'CB$ .
				Denote by $y$ the solution to the equation
				\begin{align*}
								\dot{y} &= - \nabla g(C)\\
								y(0) &= B,
 				\end{align*}
				which integrates to $y(t) = B - t w$ for all $t$. Equation \eqref{eq:thomGradient1} ensures that for any $0\leq t \leq \bar{t}$
				\begin{align*}
				                \frac{d}{dt} (\left\langle C - B, x(t) \right\rangle) &\geq \frac{d}{dt} (\left\langle C - B, y(t) \right\rangle)
				\end{align*}
				Hence, we have for any $0\leq t \leq \bar{t}$, integrating on $[0,t]$ 
				\begin{align}
				                \left\langle C - B, x(t) \right\rangle &\geq \left\langle C - B, y(t)  \right\rangle + \left\langle C - B, x_0 - B  \right\rangle \nonumber\\
				            &\geq \left\langle C - B, y(t)  \right\rangle.
                            \label{eq:thomGradient2}
				\end{align}
				Furthermore, for all $x $ in $ [BC]$, we have 
				\begin{align*}
				    1 = \|\nabla g(x)\|^2 = \frac{1}{\|C - B\|^2}\left\langle C - B, \nabla g(x) \right\rangle^2 + \left\langle v, \nabla g(x) \right\rangle^2,
				\end{align*}
				because $v$ is orthogonal to $C-B$. The first term is maximal for $x = C$ and thus the second term is minimal for $x = C$, we have thus for all $x $ in $ [BC]$
				\begin{align}
				    0 < \left\langle \nabla g(C), v \right\rangle = \left\langle - \nabla g(C), -v \right\rangle \leq \left\langle \nabla g(x), v \right\rangle =  \left\langle -\nabla g(x), -v \right\rangle   \leq 1.
				    \label{eq:thomGradient3}
				\end{align}
				Equation \eqref{eq:thomGradient3} holds for all $x$ in $OCB$ different from $O$ by homogeneity. 
				We deduce that for all $0 \leq t \leq \bar{t}$, we have
				\begin{align*}
				                \frac{d}{dt} (\left\langle -v, x(t) \right\rangle) &\geq \frac{d}{dt} (\left\langle -v, y(t) \right\rangle)
				\end{align*}
				and by integration
				\begin{align}
				            \left\langle -v, x(t) \right\rangle &\geq \left\langle -v, y(t)  \right\rangle + \left\langle -v, x_0 - B  \right\rangle \nonumber\\
				            &= \left\langle -v, y(t)  \right\rangle.
                            \label{eq:thomGradient4}
				\end{align}
				Hence, in the coordinate system $(C - B, -v)$, which is orthogonal, for all $t $ in $ [0,\bar{t}]$, $x(t)$ has larger coordinates than $y(t)$.
				
				The trajectory $y(t)$, of equation $t \mapsto B - t w$ is the line going from $B$ to $C'$. From equations \eqref{eq:thomGradient2} and \eqref{eq:thomGradient4}, we may write for all $t$ in $ [0,\bar{t}]$, $x(t) = y(t) + \alpha (t) (C-B) + \beta(t) (-v)$ where $\alpha$ and $\beta$ are positive functions. Since $y(t)$ belongs to the line $(BC')$, this shows that $x(t)$ has to be above this line for all $t \geq 0$, $t\leq \bar{t}$ and actually, $x(\bar{t}) $ in $ BCC'$. Hence at time $\bar{t}$, we have $x(\bar{t}) $ in $ [CC']$. This holds true because $x(\bar{t})$ is on the boundary of $OCB$ and on the boundary of $BCC'$. Hence either $x(\bar{t}) $ in $ [CC']$, either $x(\bar{t}) $ in $ [BC]$. Equation  \eqref{eq:thomGradient3} ensures that if $x(\bar{t}) $ in $ [BC]$ then $x(\bar{t}) = C$ which concludes the proof.
\end{proof}

\begin{corollary}[Secants of gradient curves may all fail to converge]
				There exists a $C^k$ strictly convex function on $\RR^2$ with a unique minimizer $\bar{x}$, such that any nonconstant solution to the gradient flow equation
				\begin{align*}
								\dot{x}(t) = -\nabla f(x(t))
				\end{align*}
				is such that
				\begin{align*}
								\frac{x(t) - \bar{x}}{\|x(t) - \bar{x}\|}
				\end{align*}
				does not have a limit as $t \to \infty$.\\
				The function $f$ has a positive definite Hessian  everywhere except at $0$.
				\label{cor:Thom}
\end{corollary}
\begin{proof}
				We assume without loss of generality that $\bar{x} = O$ is the origin. Writting $x(t) = (r(t),\theta(t))$ in polar coordinate, we will construct a function $f$ such that each solution to the ODE produces nonconverging trajectories $\theta(t)$. 
				
				We start with an interpolating set $S_0 = ABCDE$ as in Lemma \ref{lem:techThom} and  let $S_1 = A'B'C'D'E'$  be its scaled version as described in Figure~\ref{fig:thom}. 
				
				Let $\alpha$  be the value of the angle $\widehat{BOC}$ and $m = \left\lceil \frac{2 \pi}{\alpha} \right\rceil + 1$. We have
				\begin{align*}
				    \frac{2\pi}{m} < \alpha.
				\end{align*}
				To obtain $S_{2}$, we rotate $S_0$ by an angle $2\pi / m$, we denote $S'_0$ the resulting set. We rescale $S'_0$ by a factor $\beta $ in $ (0,1)$ so  that $\beta S'_0$ lies in the interior of $S_1$. Call the resulting set $S_2$ and $S_3$ is obtained from $S_2$ exactly the same way as $S_1$ is obtained from $S_0$. We repeat the same process indefinitely to obtain a strictly decreasing sequence of $C^k$ sets. Note that for any $k $ in $ \NN$, $S_{2km}$ and $S_{2km + 1}$ are homothetic to $S_0$. 
				
				We now invoke Corollary \ref{cor:polygonDecrease} (with $I = \ZZ$ and revert the indices) to obtain a $C^k$ function $f$ with those prescribed level sets. Using Remark \ref{rem:alignedLevelSets} it turns out that the level sets of $f$ between $S_0$ and $S_1$ are simple scalings of $S_0$. Hence the gradient curves of $f$ and those of the gauge function of $S$ are the same between $S_0$ and $S_1$, up to time reparametrization. 

				Using Lemma \ref{lem:techThom} any trajectory crossing $[BC]$ in Figure \ref{fig:thom}, must also be crossing $[CC']$ and leave the triangle $BOC$ in finite time. The same statement holds after scaling the level sets and since for all $k $ in $ \NN$, $S_{2km}$ and $S_{2km+1}$ are homothetic to $S_0$, this shows that no solution stays indefinitely in the triangle $BOC$.
				
				Lemma \ref{lem:techThom} still holds after rotations and by our construction, for any triangle $T$ obtained by rotating $BOC$ by a multiple of $2\pi/m$, no trajectory stays indefinitely  within $T$. Since $2\pi / m < \alpha$, the union of these triangles $U$ contains $O$ in its interior. 
				
			Note first that any gradient curve converges to $\bar x$. Let us argue by contradiction and assume that there exists a continuous gradient curve $t \mapsto z(t)$  distinct from the stationary solution  $\bar{x}$, such that 
                \begin{align*}
                    \frac{z(t) - \bar{x}}{\|z(t) - \bar{x}\|}
                \end{align*}				
converges. This exactly means that the angle  $\theta(t)$ of the curve has a limit in $[0,2\pi)$ as $t$ goes to infinity. There is a rotation of $BOC$ by a multiple of $2\pi/m$ whose  interior intersects  the half line given by the direction $\theta$, call this triangle $T$. The directional convergence entails that there exists $t_0 \geq 0$ such that $z(t)$ belongs to $T$ for all $t \geq t_0$. Hence $z$ can not be a gradient curve. To complete the proof, we may add disks of increasing size to the list of sets to obtain a full domain function and invoke  Theorem \ref{th:smoothinterp} with $I = \ZZ$.
\end{proof}

\subsection{Newton's flow may not converge}
Given a twice differentiable convex function  $f$, we define the open set $\Omega:=\{x\in\R^2:\nabla^2f \mbox{ is invertible}\}$ and we consider maximal solutions to the differential equation
\begin{align}
     \dot{x}(t) = - \nabla^2 f(x(t))^{-1} \nabla f(x(t)),
     \label{eq:newton}
\end{align}
on $\Omega$. This is the continuous counterpart of Newton's method, it has been studied in \cite{aubin1984differential} and \cite{alvarez1998dynamical}. Let $x_0$ be in $\Omega$, there exists a unique maximal nontrivial interval $I$ containing $0$ and a unique solution $x$ to \eqref{eq:newton} on $I$ with $x(0) = x_0$. Equation \eqref{eq:newton} may be rewritten as 
\begin{align*}
    \frac{d}{dt} \nabla f(x(t)) = - \nabla f(x(t))
\end{align*}
and thus for all $t $ in $ I$, we have
\begin{align}
    \nabla f(x(t)) = e^{-t} \nabla f(x_0).
    \label{eq:newtonIntegrated}
\end{align}
If we could ensure that $I = \RR$ and $f$ has oscillating gradients close to its minimum, then \eqref{eq:newtonIntegrated} entails that the direction of the gradient is constant along the solution, which requires oscillations in space to compensate for gradient oscillations. 

\begin{corollary}[A bounded Newton's  curve that oscillates at infinity]
				For any $k\geq2$, there exists a $C^k$  convex coercive function $f \colon \RR^2 \mapsto \RR$ and an initial condition $x_0 $ in $ \RR^2$ such that the solution to \eqref{eq:newton} is bounded, defined on $\RR$ and has at least two distinct accumulation points.\label{cor:newton}
\end{corollary}
The counterexample is sketched in Figure~\ref{fig:illustrNewton}, the construction is the same as for Corollary~\ref{cor:altMin} but instead of doing quarter rotations, we use symmetry with respect to the first axis. We can then call for Corollary~\ref{cor:polygonDecrease} to construct the function $f$ and equation \eqref{eq:newtonIntegrated} ensures that the solution interval is unbounded.

\begin{figure}[]
				\centering
				\includegraphics[width=.45\textwidth]{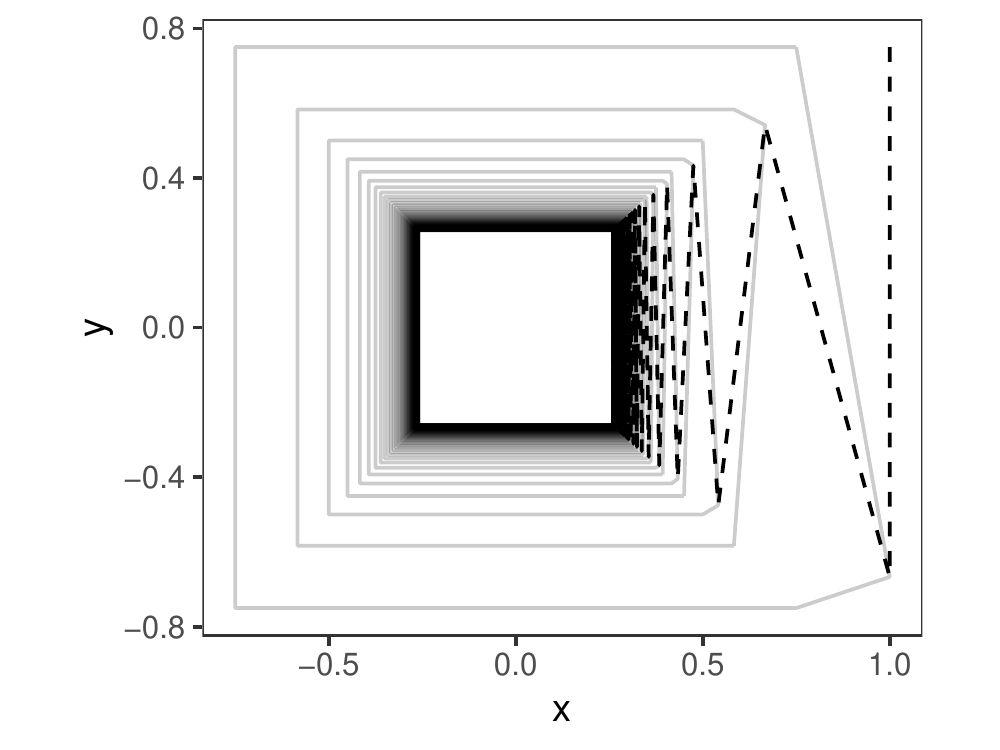}\quad \includegraphics[width=.4\textwidth]{illustrAltMin2}
				\caption{Illustration of the continuous time Newton's dynamics. On the left, the ``skeletons" of the sublevel sets in gray and a sketch of the corresponding curve. On the right,  the normals to be chosen in order to apply Lemma \ref{lem:polyGon}.}
				\label{fig:illustrNewton}
\end{figure}

\subsection{Bregman descent (mirror descent) may not converge}
The mirror descent algorithm was introduced in \cite{newmirovsky1983problem} as an efficient method to solve constrained convex problems. In \cite{beck2003mirror}, this method is shown to be equivalent to a projected subgradient method, using non-Euclidean projections. It plays an important role for some categories of constrained  optimization problem; see e.g., \cite{bauschke2016descent} for recent developments and \cite{Dragomir19} for a surprising example.

Let us recall beforehand some definitions. Given a Legendre function $h$ with domain $\dom h$, define the {\em Bregman distance}\footnote{It is actually not a proper distance.}  associated to $h$ as  $D_h(u,v)=h(u)-h(v)-\langle \nabla h(v),u-v\rangle$ where  $u$ is in $ \dom h$ and $v$ is in the interior of $\dom h$.

Given a smooth convex function $f$ that we wish to minimize on $\overline{ \dom h}$, we consider the Bregman method 
$$x_{i+1}=\argmin\left\{\langle \nabla f(x_i),u-x_i\rangle+\lambda D_h(u,x_i):u\in \R^2\right\},$$
where $x_0$ is in $\inte \dom h$ and $\lambda>0$ is a step size. When the above iteration is well defined, e.g. when $\dom h$ is bounded, it writes:
$$x_{i+1}=\nabla h^*\left(\nabla h(x_i)-\lambda \nabla f(x_i)\right).$$
In \cite{bauschke2016descent} the authors identified a generalized smoothness condition  which confers good minimizing properties to the above method:
\begin{align}
\label{hsmooth}Lh-f \mbox{ convex},\\ \lambda\in(0,L).\label{step}
\end{align}

The corollary below shows that  such an  algorithm may not converge, even though we assume the cost to satisfy \eqref{hsmooth}, the step to satisfy \eqref{step}, and the Legendre function to have a compact domain.

\begin{corollary}[Bregman or mirror descent may not converge]
				There exists  a Legendre function $h\colon D \mapsto \RR$, defined on a closed square $D$, continuous on $D$,  a vector $c$ in $ \RR^2$, and $x_0$ in $\R^2$ such that  the Bregman  descent recursion
\begin{align*}
					x_{i+1} = \nabla h^*\left( \nabla h(x_i) - c  \right),
				\end{align*}
				produces a bounded sequence $(x_i)_{i \in \NN}$ which has at least two distinct accumulation points.
				\label{cor:mirrorDescent}
\end{corollary}
\begin{proof}
                We fix $\theta \in \left( \frac{-\pi}{4},\frac{\pi}{4} \right)$, $\theta \neq 0$, and consider $h$ constructed in Corollary \ref{cor:legendreOscillating} and choose $c = (-1,0)$. In this case  the Bregman  descent recursion writes for all $i$ in $\NN$,
				\begin{align*}
								\nabla h (x_{i+1}) - \nabla h(x_i) = -c  
				\end{align*}
				so that we actually have $\nabla h (x_{i}) - \nabla h(x_{0}) = \nabla h (x_{i}) = - i c$
				and thus  $$x_i = \nabla h^*(-ic) = \nabla h^*(i,0).$$ 
				By Corollary ~\ref{cor:legendreOscillating}, we have for all $i \in \NN$ that $\nabla h^*(i,0)$ proportional to $(\cos(\theta), (-1)^i \sin(\theta))$. Since the norm of the gradient of $h^*$ cannot vanish at infinitiy (no flat direction) and is bounded, this proves that the sequence $(x_i)_{i \in \NN}$ has at least two accumulation points which is the desired result.
\end{proof}

\subsection{Central paths of Legendre barriers may not converge}
Consider the problem
\begin{align}
				\min_{x \in D} \left\langle c, x\right\rangle
				\label{eq:LP}
\end{align}
where $D$ is a subset of $\RR^2$ and $c$. Given a Legendre function $h$ on $D$, we introduce the $h$ central path  through 
\begin{align}
				x(r) = \argmin\left\{ \left\langle c, x\right\rangle + r h(x):x\in\R^2\right\}
				\label{eq:interiorLegendre}
\end{align}
where $r>0$ is meant to tend to $0$. Central paths are one of the essential tools behind interior point methods, see e.g., \cite{NN,Aus99} and references therein.

Note that the accumulation points of $x(r)$ as $r \to 0$, have to be in the the solution set of  \eqref{eq:LP}. It is even  tempting to think that the  convergence of the path to some specific minimizer could occur, as it is the case for many barriers, see e.g. \cite{Aus99}. We have however:

\begin{corollary}[A nonconverging central path]
			There exists  a Legendre function $h\colon D \mapsto \RR$, defined on a closed square $D$, continuous on $D$,  a vector $c$ in $ \RR^2$,  such that the $h$ central path  $r \mapsto x(r)$ has two distinct accumulation points.
				\label{cor:interiorPoint}
\end{corollary}
\begin{proof}
				The optimality condition which characterizes $x(r)$ for any $r > 0$ writes,
				\begin{align*}
								x(r)= \nabla h^*\left( \frac{c}{r} \right),
				\end{align*}
				and the construction is the same as in Corollary \ref{cor:mirrorDescent}.
\end{proof}

\subsection{Hessian Riemannian gradient dynamics may not converge}
The construction of this paragraph is similar to  the two previous paragraphs.
Consider a $C^k$ ($k\geq 2$) Legendre function $h \colon D \mapsto \RR$ and the continuous time dynamics
\begin{align}
    \dot{x}(t) = - \nabla_H f(x(t)), \, t\geq 0,
    \label{eq:hessianRiemannian}
\end{align}
where $H = \nabla^2 h$ is the Hessian of $h$ and $\nabla_H f = H^{-1}  \nabla f$ is the gradient of some differentiable function $f$ in the Riemannian metric induced by $H$ on $\inte D$. Such dynamics were considered in \cite{BT03,alvarez2004hessian}.

We have the following result:
\begin{corollary}[Nonconverging Hessian Riemannian gradient dynamics]
			There exists  a Legendre function $h\colon D \mapsto \RR$, defined on a closed square $D$, continuous on $D$,  a vector $c$ in $ \RR^2$, and $x_0$ in $\R^2$ such that  the solution to \eqref{eq:hessianRiemannian} with $f=\langle c,\cdot\rangle$ has two distinct accumulation points.
				\label{cor:hessianRiemannian}
\end{corollary}
\begin{proof}
				Equation \eqref{eq:hessianRiemannian} may be rewritten
				\begin{align*}
								\frac{d}{dt} \nabla h(x(t))= - \nabla f(x(t)),
				\end{align*}
				so choosing $c = (-1,0)$, we have for all $t \in \RR$, $\nabla h(x(t)) = \nabla h(x(0)) + (t,0) = (t,0)$ and the construction is the same as in Corollary \ref{cor:mirrorDescent}.
\end{proof}

\section{Appendix}

\begin{lemma}[Smooth concave interpolation: in between square root and affine]
				There exists a $C^\infty$ strictly increasing concave function $\phi\colon [0,1] \mapsto [0,1]$ such that 
				\begin{align*}
								\phi(t) &= \sqrt{2t/3} \quad \forall t \leq 1/6\\
								\phi(1) &= 1 \\
								\phi'(1) &= 2/3\\
								\phi^{(m)}(1) &= 0, \quad \forall m \geq 2
				\end{align*}
				\label{lem:interpolationAroundZero}
\end{lemma}

\begin{proof}
				Consider a $C^\infty$ function $g_0 \colon \RR \mapsto [0,1]$ such that $g_0 = 1$ on $(-\infty,-1)$, $g_0 = 0$ on $(1, +\infty)$ (for example convoluting the step function with a smooth bump function). Set $g(t) = \frac{1}{2}\left( g_0(t) + 1 - g_0(-t) \right)$ we have that $g$ is $C^\infty$,  $g = 1$ on $(-\infty,-1)$, $g = 0$ on $(1, +\infty)$ and $g(t) + g(-t) = 1$ for all $t$. We have
				\begin{align*}
								\int_{-1}^1 g(s) ds = 1\\
								\int_{-1}^1 \left( \int_{-1}^t g(s)ds \right) dt = 1
				\end{align*}
				Set $\phi_0 \colon [-3,3] \mapsto \RR$, such that
				\begin{align*}
								\phi_0(t) = \int_{-3}^t \left( \int_{-3}^r g(s) ds\right)dr.
				\end{align*}
				For all $r $ in $ [-3,3]$, we have
				\begin{align*}
								\int_{-3}^r g(s) ds = 
								\begin{cases}
												r+3 & \text{ if } r \leq -1\\
												2 + \int_{-1}^r g(s)ds & \text{ if } -1 \leq r \leq 1\\
												3 & \text{ if } r \geq 1
								\end{cases}
				\end{align*}
				and thus
				\begin{align*}
								\phi_0(t) = 
								\begin{cases}
												\frac{t^2}{2} - 9/2 + 3(t+3) & \text{ if } t \leq -1\\
												2 + 2(t + 1) + \int_{-1}^t \left( \int_{-1}^r g(s)ds \right)dr& \text{ if } -1 \leq t \leq 1\\
												6 + 3(t-1) & \text{ if } 1 \geq t
								\end{cases}
				\end{align*}
				and in particular $\phi_0(3) = 12$ and $\phi_0'(3) = 3$. Set $\phi_1(s) = \phi_0(6 s -3)/12$. 
				\begin{align*}
								\phi_1(0) &= 0\\
								\phi_1(t) &= \left(\frac{(6t-3)^2}{2} - 9/2 + 2(3t )\right)/12 =  3 t^2 / 2 =  \text{ if } t \leq 1/3\\
								\phi_1(1) &= 1\\
								\phi_1'(1) &= 3/2.
				\end{align*}
				$\phi_1$ is stricly increasing, let $\phi \colon [0,1] \mapsto [0,1]$ denote the inverse of $\phi_1$, we have 
				\begin{align*}
								\phi(1) &= 1\\
								\phi'(1) & = 2 / 3\\
								\phi(t) &=  \sqrt{2t/3} \text{ if } t\leq 1/6.
				\end{align*}
\end{proof}

\begin{lemma}[Interpolation inside a sublevel set]
				Consider any strictly increasing $C^k$ function $\phi \colon (0,2) \mapsto \RR$ such that $\phi(1) = 1$ and $\phi^{(m)}(1) = 0$, $m = 2,\ldots k$. Then the function
				\begin{align*}
								G \colon (0,2) \times \RR / 2 \pi \Z  & \mapsto \RR^2 \\
								(s,\theta) &\mapsto \phi(s) n(\theta)
				\end{align*}
				is diffeomorphism which satisfies for any $m=1 \ldots,k$ and $l =2,\ldots, k$,
				\begin{align*}
								&\frac{\partial^m G}{\partial \theta^m}(1,\theta) = n^{(m)}(\theta)\\
								&\frac{\partial^{m+1} G}{\partial \lambda\partial \theta^m} (1,\theta) = \phi'(1)n^{(m)}(\theta) \\
								&\frac{\partial^{l+m} G}{\partial \lambda^l \partial \theta^m }(\lambda_-,\theta) = 0.
				\end{align*}
				\label{lem:diffGauge2}
\end{lemma}

\begin{lemma}{\bf: Combinatorial Arbogast-Fa\`a di Bruno Formula (from \cite{ma2009higher}).}
				Let $g \colon \RR \mapsto \RR$ and $f \colon \RR^p \mapsto [0, +\infty)$ be $C^k$ functions. Then we have for any $m \leq k$ and any indices $i_1,\ldots,i_m \in \left\{ 1,\ldots, p \right\}$.
				\begin{align*}
								\frac{\partial^m}{\prod_{l=1}^{m}\partial x_{i_l}} g \circ f(x) = \sum_{\pi \in \PP} g^{(|\pi|)}(f(x)) \prod_{B \in \pi} \frac{\partial^{|B|} f}{\prod_{l \in B}\partial x_{i_l}}(x),
				\end{align*}
				where $\PP$ denotes all partitions of $\left\{ 1,\ldots, m \right\}$, the product is over subsets of $\left\{ 1,\ldots,m \right\}$ given by the partition $\pi$ and $|\cdot|$ denotes the number of elements of a set. We rewrite this as follows
				\begin{align*}
								\frac{\partial^m}{\prod_{l=1}^{m}\partial x_{i_l}} g \circ f(x) = \sum_{k = 1}^m\sum_{\pi \in \PP_k} g^{(k)}(f(x)) \prod_{B \in \pi} \frac{\partial^{|B|} f}{\prod_{l=1}^{m}\partial x_{i_l}}(x),
				\end{align*}
				where $\PP_k$ denotes all partitions of size $k$ of $\left\{ 1,\ldots, m \right\}$.
				\label{lem:faaDiBruno}
\end{lemma}

\begin{lemma}{From \cite[Lemma 45]{bolte2010characterization}}
				Let $h $ in $ C^0\left( (0,r_0],\RR_+^* \right)$ be an increasing function. Then there exists a function $\psi $ in $ C^\infty(\RR,\RR_+)$ such that $\psi = 0$ on, $\RR_-$ and $0 < \psi(s) \leq h(s)$ for any $s$ in $(0,r_0]$ and $\psi$ is increasing on $\RR$	
				\label{lem:CinfiniteLowerBound}
\end{lemma}

\begin{lemma}[High-order smoothing near the solution set]
				Let $D \subset \RR^p$ be a nonempty compact convex set and $f \colon D \mapsto \RR$ convex, continuous on $D$ and $C^k$ on $D \setminus \argmin_{D} f$. Assume further that $\argmin_D f \subset \mathrm{int}(D)$, $k \geq 1$, with $\min_D f = 0$. Then there exists $\phi \colon \RR \mapsto \RR_+$, $C^k$, convex and increasing with positive derivative on $(0,+\infty)$, such that $\phi\circ f$ is convex and $C^k$ on $D$.
				\label{lem:CkSmoothing}
\end{lemma}
\begin{proof}
				By a simple translation, we may assume that $\min_D f = 0$ and $\max_D f = 1$.
				Any convex function is locally Lipschitz continuous on the interior of its domain so that $f$ is globally Lipschitz continuous on $D$ and its gradient is bounded. Hence, $f^2$ is $C^1$ and convex on $D$. 
				We now proceed by recursion. For any $m =1,\ldots, k$, we let $Q_m$ denote the $m$-order tensor of partial derivatives of order $m$. Fix $m$ in $\{1,\ldots,k\}$. Assume that $f$ is $C^m$ throughout  $D$ while it is $C^{m+1}$ on $D \setminus \arg\min_D f$. Note that all the derivatives up to order $m$ are bounded. We wish to prove that $f$ is globally $C^{m+1}$.

				Consider the increasing function
				\begin{align*}
								h \colon (0,1] &\mapsto \RR_+^*\\
								s &\mapsto \frac{s}{1 + \sup_{s \leq f(x) \leq 1}\|Q_{m+1}(x)\|_{\infty}}
				\end{align*}
				and set $\psi$ as in Lemma \ref{lem:CinfiniteLowerBound}. Recall that $\psi$ is $C^\infty$ and all its derivative vanish at $0$ and $\psi \leq h$ on $(0,1]$. Let $\phi$ denote the anti-derivative of $\psi$ such that $\phi(0) = 0$. $\phi$ is $C^\infty$ and convex increasing on $\RR$ and, since its  derivatives at $0$ vanish as well, one has, for any $q $ in $ \NN$, $\phi^{(q)}(z) = o(z)$. Consider the function $\phi \circ f$. It is $C^m$ on $D$ and it has bounded derivatives up to order $m$. Furthermore, it is $C^{m+1}$ on $D \setminus \argmin_D f$. Let $\bar{y} $ in $ \argmin_D f$. If $\bar{y} $ in $ \mathrm{int}(\argmin_D f)$, then $f$ and $\phi \,\circ f$ have derivatives of all order vanishing at $\bar{y}$. Assuming that $\bar{y} $ in $ \argmin_D f\setminus \mathrm{int}(\argmin_D f)$. 
				By the induction assumption and Lemma~\ref{lem:faaDiBruno}, we have for any indices $i_1,\ldots,i_m \in \left\{ 1,\ldots, p \right\}$ and any $h $ in $ \RR^p$:
				\begin{align*}
								&\frac{\partial^m}{\prod_{l=1}^{m}\partial x_{i_l}} (\phi \circ f)(\bar{y} + z) - \frac{\partial^m}{\prod_{l=1}^{m}\partial x_{i_l}} (\phi \circ f)(\bar{y}) \\
								=\;& \frac{\partial^m}{\prod_{l=1}^{m}\partial x_{i_l}}( \phi \circ f)(\bar{y} + z) \\
								=\;&\sum_{q = 1}^{m}\sum_{\pi \in \PP_q} \phi^{(q)}(f(\bar{y} + z)) \prod_{B \in \pi} \frac{\partial^{|B|} f}{\prod_{l=1}^{m}\partial x_{i_l}}(\bar{y} + z).
				\end{align*}
			 All the derivatives of $f$ are of order less or equal to $m$ and thus remain bounded as $z \to 0$. Further more $f$ is Lipschitz continuous on $D$ so that  $f(\bar{y} + z) = O(\|z\|)$ near $0$, and, for any $q $ in $ \NN$,  $\phi^{(q)}(f(\bar{y} + z)) = o(\|z\|)$. Hence $\phi \circ f$ has derivative of order $m+1$ at $\bar{y}$ and it is $0$.

				Since $\argmin_D f \subset \mathrm{int}(D)$, we may consider any sequence of point $(y_{j})_{j \in \NN}$ in $D \setminus \argmin_D f$ converging to $\bar{y}$. 
				By Lemma \ref{lem:faaDiBruno}, we have for any indices $i_1,\ldots,i_{m+1} \in \left\{ 1,\ldots, p \right\}$, and any $j $ in $ \NN$,
				\begin{align*}
								\frac{\partial^{(m+1)}}{\prod_{l=1}^{m+1}\partial x_{i_l}} (\phi \circ f)(y_j) &= \phi'(f(y_j)) \frac{\partial^{(m+1)} f}{\prod_{l=1}^{m}\partial x_{i_l}}(y_j) + \sum_{q = 2}^{m+1}\sum_{\pi \in \Pi_q} \phi^{(q)}(f(y_j)) \prod_{B \in \pi} \frac{\partial^{|B|} f}{\prod_{l=1}^{m}\partial x_{i_l}}(x)\\
								&\leq h(f(y_j))\frac{\partial^{(m+1)} f}{\prod_{l=1}^{m}\partial x_{i_l}}(y_j) + \sum_{q = 2}^{m+1}\sum_{\pi \in \Pi_q} \phi^{(q)}(f(y_j)) \prod_{B \in \pi} \frac{\partial^{|B|} f}{\prod_{l=1}^{m}\partial x_{i_l}}(x)\\
								&= f(y_j) \frac{\frac{\partial^{(m+1)} f}{\prod_{l=1}^{m}\partial x_{i_l}}(y_j)}{1 + \sup_{f(y_j) \leq f(x) \leq 1}\|Q_{m+1}(x)\|_{\infty}}  + O(f(y_j))\\
								&= O(f(y_j)),
				\end{align*}
				where  the inequality follows from the construction of $\phi$. The third step follows using the definition of $h$ and the fact that, for any $q \geq 2$, 
				\begin{enumerate}
				    \item Each partition of $\left\{ 1,\ldots,m+1 \right\}$ of size $q$ contains subsets of size at most $m$. Thus in the product, the terms $\partial^{|B|} f$ correspond to bounded derivatives of $f$ by the induction hypothesis.
				    \item  $\phi^{(q)}(a) = o(a)$ as $a \to 0$.
				\end{enumerate}
				The last step stems from the fact that the ratio has asbolute value less than $1$.
				This shows that the derivatives of order $m+1$ of $\phi \circ f$ are decreasing to $0$ as $j \to \infty$ and $\phi \circ f$ is actually $C^{m+1}$ and convex on $D$. The result follows by induction up to $m = k$ and by the fact that a  composition of increasing convex functions is increasing and convex.
\end{proof}

\begin{lemma}
                Let $p \colon \RR_+ \mapsto \RR_+$ be concave increasing and $C^1$ with $p' \geq c$ for some $c > 0$. Assume that there exists $ A > 0$ such that for all $x $ in $ \RR_+$
                \begin{align*}
                    p(x) - x p'(x) \leq A.
                \end{align*}
                Then setting $a= A/c$, we have for all $x \geq a$,
                \begin{align*}
                    p(x-a) - x p'(x-a) \leq 0
                \end{align*}
                \label{lem:techConcaveAsymptotic}
\end{lemma}
\begin{proof}
    For all $x \geq a$, we have
    \begin{align*}
        f(x-a) - (x-a)f'(x-a) \leq A,
    \end{align*}
 hence 
    \begin{align*}
        f(x-a) - xf'(x-a) \leq A - af'(x-a) \leq A - ac = 0.
    \end{align*}
\end{proof}

\noindent
{\bf Acknowledgements.} The authors acknowledge the support of AI Interdisciplinary Institute ANITI funding, through the French ``Investing for the Future -- PIA3'' program under the Grant agreement n°ANR-19-PI3A-0004, Air Force Office of Scientific Research, Air Force Material Command, USAF, under grant numbers FA9550-19-1-7026, FA9550-18-1-0226, and ANR MasDol. J. Bolte acknowledges the support of ANR Chess, grant ANR-17-EURE-0010 and ANR OMS.

\end{document}